\documentclass[11pt]{article}

\usepackage[T1]{fontenc}
\usepackage[utf8]{inputenc}
\usepackage{amsmath,amssymb,mathtools,amsthm}
\usepackage{booktabs,enumitem,graphicx}
\usepackage[margin=1in]{geometry}
\usepackage{authblk}
\usepackage[hidelinks]{hyperref}
\newtheorem{theorem}{Theorem}[section]
\newtheorem{proposition}[theorem]{Proposition}
\newtheorem{lemma}[theorem]{Lemma}
\newtheorem{corollary}[theorem]{Corollary}
\newtheorem{remark}[theorem]{Remark}

\title{Asymptotic minimization of the first Dirichlet eigenvalue in a square with two hard obstacles}

\author[1]{Baruch Schneider\thanks{Corresponding author: \texttt{baruch.schneider@osu.cz}}}
\author[1]{Diana Schneiderov\'a}
\author[1,2,3]{Yifan Zhang}
\affil[1]{Department of Mathematics, University of Ostrava, Ostrava, Czech Republic}
\affil[2]{Department of Applied Mathematics, V\v{S}B--Technical University of Ostrava, Ostrava, Czech Republic}
\affil[3]{Department of Algebra, Charles University, Prague, Czech Republic}
\date{}

\begin{document}

\maketitle

\begin{abstract}
We study an asymptotic minimization problem for the first Dirichlet eigenvalue in the square $Q=(-1,1)^2$, where two equal hard circular obstacles of radius $r$ move inside $Q$ with disjoint interiors; tangency between the obstacles and with the outer boundary is allowed. We prove that, as $r\to0$, every minimizing configuration approaches, up to the symmetries of the square and interchange of the holes, a pair of corner-tangent obstacles at adjacent corners, with displacement $o(r)$ from the corresponding corner-tangent centers. A uniform $u$-capacity localization argument shows that any obstacle whose individual eigenvalue excess is $O(r^4)$ must lie in a fixed $O(r)$ neighborhood of a corner, thereby excluding interior, open-side, and intermediate approaches to the boundary. At the corner scale, the leading coefficient is the full $u$-capacity. The proof is then organized by geometric branches. The side-tangent one-hole branch is asymptotically driven to the adjacent corner; an exact polarization argument shows that the adjacent corner-tangent pair has smaller first eigenvalue than the opposite pair; and configurations concentrated near a single corner are excluded by a scalar one-hole comparison. Additivity for well-separated corner defects then completes the distinct-corner analysis. We also include reproducible finite element computations that support the branch ordering predicted by the proof.
\end{abstract}



\section{Introduction}\label{sec:introduction}

Spectral shape optimization for elliptic operators is a classical subject at the intersection of analysis, geometry, and optimization; see, for instance, Henrot~\cite{Henrot2006}, Belhachmi--Bucur--Buttazzo--Sac-Ep\'ee~\cite{BelhachmiEtAl2006}, and Buttazzo's survey~\cite{Buttazzo2011}. The first Dirichlet eigenvalue is one of the most basic functionals in this area. In mechanical and acoustic models it represents the fundamental frequency of a clamped membrane or cavity, in diffusion it controls the slowest exponential decay rate, and in quantum confinement it gives the ground-state energy under hard-wall constraints. Eigenvalue-driven design problems therefore arise naturally in vibration-sensitive structural design, topology optimization, and related material-layout problems; see, for example, Bends\o e--Sigmund~\cite{BendsoeSigmund2004}, Seyranian--Lund--Olhoff~\cite{SeyranianLundOlhoff1994}, and Garcke--H\"uttl--Knopf~\cite{GarckeHuettlKnopf2022}.

In this paper the ambient domain is fixed and the optimization variable is the placement of internal hard obstacles. This viewpoint goes back, for the fundamental Dirichlet eigenvalue, to Harrell--Kr\"oger--Kurata~\cite{HKK2001}, who showed for certain domains that a minimizing obstacle must touch the outer boundary, while in sufficiently symmetric settings maximizing positions lie in the interior, typically at the center. Since then, obstacle-placement questions have been investigated in several complementary directions. El Soufi--Kiwan~\cite{ElSoufiKiwan2008} treated doubly connected planar domains with dihedral symmetry, while Kiwan--El Soufi~\cite{KiwanElSoufi2008} studied the second Dirichlet eigenvalue for spherical obstacles. Related placement principles were extended from the first eigenvalue to other spectral functionals, such as the Dirichlet heat trace in El Soufi--Harrell~\cite{ElSoufiHarrell2016} and the Dirichlet heat content in Li~\cite{Li2022HeatContent}. Variants with different boundary conditions have also been analyzed, for example for the first nonzero Steklov eigenvalue in Ftouhi~\cite{Ftouhi2022} and for Steklov problems with an internal Dirichlet obstacle in Paoli--Piscitelli--Sannipoli~\cite{PaoliPiscitelliSannipoli2021}. A complementary localization result for obstacle maximizers was obtained by Georgiev--Mukherjee~\cite{GeorgievMukherjee2018}. From a more asymptotic multi-inclusion viewpoint, Maz'ya--Movchan--Nieves~\cite{MazyaMovchanNieves2017} analyzed the spectral effect of solids containing many small inclusions.

The problem considered here is the next natural step on the minimization side: two equal hard obstacles in a square. Besides its intrinsic theoretical interest, this two-obstacle problem is a clean prototype for the arrangement of rigid inclusions, voids, blocked regions, or impenetrable defects in a confined medium. It is also a tractable model for a genuinely multi-shape optimization problem, where interaction between obstacles and boundary singularities matters already at leading order. In a polygonal domain such as a square, the small-hole regime is especially delicate because corners create a competing $r^4$ scale that is strictly better than the $r^2$ scale associated with a flat side. Indeed, the Dirichlet ground state vanishes to first order at an open side but to second order at a square corner: after the appropriate odd reflections, the first nonzero homogeneous profiles have degrees $1$ and $2$, respectively. The two-dimensional simple-eigenvalue $u$-capacity expansion therefore scales like $r^{2k}$ for a degree-$k$ profile (see in particular Theorem~6.2 of~\cite{ABLM2021}), which gives the orders $r^2$ and $r^4$.

Throughout the paper,
\[
Q=(-1,1)^2.
\]

For $r>0$ set
\[
Q_r:=[-1+r,1-r]^2,
\qquad
\mathcal C_r:=\Bigl\{(x_1,x_2)\in Q_r\times Q_r:\ |x_1-x_2|\ge 2r\Bigr\},
\]
and
\[
\Lambda_r(x_1,x_2):=\lambda_1\Bigl(Q\setminus\bigl(\overline{B_r(x_1)}\cup \overline{B_r(x_2)}\bigr)\Bigr).
\]
Thus the closed disks may touch each other or $\partial Q$, but their interiors do not overlap. This explicit compact center set is the admissible class intended throughout the paper.

Before stating the main theorem, we record continuity at tangency and existence of minimizers in the compact admissible class.

\begin{lemma}[Existence and continuity at tangency]\label{lem:existence-continuity}
For every $r>0$ for which $\mathcal C_r$ is nonempty, the map $\Lambda_r$ is continuous on the compact set $\mathcal C_r$. In particular, the global minimum and the restricted minima used below are attained.
\end{lemma}

\begin{proof}
Let $(x_{1,n},x_{2,n})\to(x_1,x_2)$ in $\mathcal C_r$ and let $K_n$ and $K$ be the corresponding unions of closed disks. We verify Mosco convergence of the zero extensions of $H^1_0(Q\setminus K_n)$ to $H^1_0(Q\setminus K)$ inside $H^1_0(Q)$. If $\varphi\in C_c^\infty(Q\setminus K)$, Hausdorff convergence $K_n\to K$ makes $\varphi$ admissible for all large $n$; density and a diagonal argument give the recovery condition. Conversely, if $u_n\in H^1_0(Q\setminus K_n)$ converges weakly in $H^1_0(Q)$ to $u$, then $u=0$ almost everywhere in the interior of each limiting disk. Hence its quasi-continuous representative vanishes quasi-everywhere on each closed disk. The only additional points created by disk--disk or disk--boundary tangency are isolated and have zero planar $H^1$-capacity. Thus $u\in H^1_0(Q\setminus K)$. The min--max principle now gives continuity of the first eigenvalue, and compactness of $\mathcal C_r$ gives attainment.
\end{proof}

Here and throughout the paper, a hole is called \emph{corner-tangent} if it is tangent to the two sides meeting at a corner of the square. We use \emph{asymptotically corner-tangent} when the center differs from the corresponding exact corner-tangent center by $o(r)$.

We can now state the main result.
\begin{theorem}[Main theorem]\label{thm:main-small-hole-minimizer-v12}
Let $(x_{1,r},x_{2,r})\in\mathcal C_r$ minimize $\Lambda_r$ over $\mathcal C_r$. Then, after relabeling the two centers and applying a symmetry of the square if needed,
\[
x_{1,r}=(-1+r,1-r)+o(r),
\qquad
x_{2,r}=(1-r,1-r)+o(r)
\qquad (r\to0).
\]
Equivalently, as $r\to0$, every global minimizing configuration approaches at the obstacle scale corner-tangent configurations at \emph{adjacent} corners; the displayed $o(r)$ statement is the precise meaning of asymptotic corner tangency. The adjacent corner pattern is unique up to the dihedral symmetries of $Q$ and interchange of the two holes.
\end{theorem}

The proof is organized by first separating global localization from the corner-scale optimization. Section~\ref{sec:global-localization} proves that a one-hole eigenvalue excess of order $O(r^4)$ forces the center into a fixed $O(r)$ corner layer; by domain monotonicity the same conclusion holds for each obstacle in any two-hole configuration with total cost $O(r^4)$. This uniform statement also excludes intermediate approaches to a corner, which are not covered by a dichotomy between ``interior'' and ``side-tangent'' placements. Once this reduction is established, only two global branches remain: either the two holes lie in the same corner layer, or they lie in two distinct corner layers.

More precisely, Theorem~\ref{thm:main-small-hole-minimizer-v12} follows from the following structural statements.
\begin{enumerate}[label=\rm(\arabic*)]
\item Section~\ref{sec:global-localization} proves the global $O(r)$ corner localization of every obstacle with eigenvalue excess $O(r^4)$.
\item In Section~\ref{sec:side-branch}, Theorem~\ref{thm:one-hole-side-branch-v5} proves that, within the side-tangent one-hole branch, the unique asymptotically optimal placement is the corner-tangent one.
\item In Section~\ref{sec:adjacent-vs-opposite}, Theorem~\ref{thm:adjacent-better-than-opposite-v6} compares the two corner-tangent two-hole patterns and shows that the adjacent pair strictly beats the opposite pair.
\item In Section~\ref{sec:same-corner}, Theorem~\ref{thm:same-corner-branch-complete-v11} excludes all same-corner clustered configurations.
\item In Section~\ref{sec:distinct-corner-proof}, Proposition~\ref{prop:separated-corner-additivity-v12} and Theorem~\ref{thm:distinct-corner-branch-v12} treat the remaining distinct-corner branch and complete the proof of the main theorem.
\end{enumerate}

We also include, in Section~\ref{sec:numerics}, a reproducible finite element validation based on the archived benchmark dataset~\cite{Zhang2026DatasetFEM}. These computations are not part of the proof, but they provide an independent numerical check of the branch ordering predicted by the analysis.

\section{Global localization of obstacles with small eigenvalue excess}\label{sec:global-localization}

This section supplies the global reduction used throughout the rest of the proof. We first prove a one-obstacle localization estimate and then pass to two obstacles by domain monotonicity. The argument distinguishes interior, open-side, and corner limits and, in the corner case, gives a quantitative bound that also rules out intermediate scales.

For a single disk set
\[
\lambda_r^{(1)}(x):=\lambda_1\bigl(Q\setminus\overline{B_r(x)}\bigr),
\qquad x\in Q_r.
\]
For a bounded Lipschitz domain $D$, a compact set $K\subset\overline D$, and $u\in H^1_0(D)$, we use the finite-domain $u$-capacity
\begin{equation}\label{eq:finite-u-capacity}
\operatorname{Cap}_D(K,u):=
\inf\left\{\int_D|\nabla V|^2:
V\in H^1_0(D),\quad V-u\in H^1_0(D\setminus K)\right\},
\end{equation}
where $H^1_0(D\setminus K)$ is understood through quasi-continuous representatives. This convention permits $K$ to touch $\partial D$. If $K$ is a finite union of disks, the capacitary potential agrees with $u$ almost everywhere in $K\cap D$, and therefore
\begin{equation}\label{eq:capacity-inside-lower}
\operatorname{Cap}_D(K,u)\ge \int_{K\cap D}|\nabla u|^2.
\end{equation}
For interior compact sets concentrating to a zero-capacity set, the simple-eigenvalue perturbation formula of Abatangelo--Felli--Hillairet--L\'ena~\cite[Theorem~1.4]{AFHL2019} gives
\begin{equation}\label{eq:AFHL-simple}
\lambda_N(D\setminus K_n)-\lambda_N(D)
=\operatorname{Cap}_D(K_n,u)+o\!\left(\operatorname{Cap}_D(K_n,u)\right)
\end{equation}
for the $L^2(D)$-normalized eigenfunction $u$ of a simple limiting eigenvalue.

When we use \eqref{eq:AFHL-simple} after reflection, the eigenvalue branch is unambiguous. After one odd reflection the fixed rectangle is $(-1,1)\times(-2,2)$, where $\pi^2/2$ is the simple eigenvalue corresponding to the mode $(1,2)$; after two odd reflections the fixed square is $(-1,3)^2$, where $\pi^2/2$ is the simple eigenvalue corresponding to $(2,2)$. Thus for sufficiently small reflected defects the odd, respectively odd--odd, extension of the original first eigenfunction is the unique full-spectrum eigenvalue branch converging to $\pi^2/2$.

\begin{lemma}[Smallness of a boundary-concentrating capacity potential]\label{lem:boundary-cap-smallness}
Let $D\subset\mathbb R^2$ be bounded, connected and Lipschitz, let $S\subset\partial D$ be finite, and let $K_n\subset\overline D$ concentrate to $S$. If $V_n$ is the $u$-capacitary potential and
\[
c_n:=\operatorname{Cap}_D(K_n,u)=\int_D|\nabla V_n|^2\to0,
\]
then $\|V_n\|_{L^2(D)}^2=o(c_n)$.
\end{lemma}
\begin{proof}
Otherwise normalize $W_n:=V_n/\|V_n\|_2$. A subsequence is bounded in $H^1_0(D)$ and converges strongly in $L^2(D)$ to a nonzero limit $W$. The Euler--Lagrange equation for $V_n$ may be tested against every compactly supported smooth function for all large $n$, because $K_n$ concentrates at the finite boundary set $S$. Thus $W$ is weakly harmonic in $D$. Since $W\in H^1_0(D)$, testing with $W$ gives $W=0$, a contradiction.
\end{proof}

\begin{proposition}[Boundary-concentrating simple-eigenvalue reduction]\label{prop:boundary-simple-reduction}
Assume in addition that $H^1_0(D\setminus K_n)$ converges to $H^1_0(D)$ in the Mosco sense and that $\lambda_N(D)$ is simple with normalized eigenfunction $u$. If $c_n>0$, $c_n\to0$, and $\|V_n\|_2=o(c_n^{1/2})$, then
\[
\lambda_N(D\setminus K_n)=\lambda_N(D)+c_n+o(c_n).
\]
\end{proposition}
\begin{proof}
Put $\psi_n=u-V_n\in H^1_0(D\setminus K_n)$. Capacity orthogonality gives, for every $\varphi\in H^1_0(D\setminus K_n)$,
\[
\int_D\nabla\psi_n\!\cdot\nabla\varphi-
\lambda_N(D)\int_D\psi_n\varphi
=\lambda_N(D)\int_D V_n\varphi.
\]
Hence $\psi_n$ is an $o(c_n^{1/2})$ quasimode in $L^2$. By Mosco convergence and the simplicity of $\lambda_N(D)$, for all sufficiently large $n$ there is a unique eigenvalue $\mu_n$ of $D\setminus K_n$ near $\lambda_N(D)$, separated uniformly from the remaining spectrum. Let $u_n$ be the corresponding $L^2$-normalized eigenfunction, with its sign chosen so that $\langle u_n,u\rangle_{L^2(D)}>0$. The spectral projection estimate for the above residual gives
\[
\|u_n-u\|_{L^2(D)}=o(c_n^{1/2}).
\]
Testing the quasimode identity with $u_n$ yields
\[
(\mu_n-\lambda_N(D))\langle\psi_n,u_n\rangle_{L^2(D)}
=\lambda_N(D)\langle V_n,u_n\rangle_{L^2(D)}.
\]
Capacity orthogonality, tested with $V_n-u$, gives
\[
c_n=\int_D\nabla u\cdot\nabla V_n
=\lambda_N(D)\int_D u\,V_n.
\]
Moreover,
\[
|\langle V_n,u_n-u\rangle|
\le \|V_n\|_2\,\|u_n-u\|_2=o(c_n),
\qquad
\langle\psi_n,u_n\rangle=1+o(1).
\]
Consequently $\mu_n-\lambda_N(D)=c_n+o(c_n)$, which proves the claim.
\end{proof}

For a corner $p=(\varepsilon_1,\varepsilon_2)\in\{\pm1\}^2$ and $R>1$, set
\[
\mathcal N_r^R(p):=
\left\{\bigl(\varepsilon_1(1-ar),\varepsilon_2(1-br)\bigr):(a,b)\in[1,R]^2\right\}.
\]

\begin{proposition}[Uniform one-hole localization]\label{prop:uniform-one-hole-localization}
For every $C>0$ there exist $R>1$ and $r_0>0$ such that, if $0<r<r_0$, $x\in Q_r$, and
\[
\lambda_r^{(1)}(x)-\frac{\pi^2}{2}\le Cr^4,
\]
then $x\in\mathcal N_r^R(p)$ for some corner $p$.
\end{proposition}
\begin{proof}
Suppose otherwise. Then for a sequence $r_n\to0$ there are centers $x_n\in Q_{r_n}$ with eigenvalue excess $O(r_n^4)$ that escape every fixed corner layer. Passing to a subsequence, $x_n\to x_*\in\overline Q$.

If $x_*\in Q$, the normalized ground state $u_0(x,y)=\cos(\pi x/2)\cos(\pi y/2)$ is nonzero at $x_*$. The interior small-disk asymptotics~\cite[Propositions~1.5--1.6]{AFHL2019} give a shift of order $1/|\log r_n|$, contradicting $O(r_n^4)$.

If $x_*$ lies on an open side, write after symmetry $x_n=(\xi_n,1-d_n)$ with $d_n\ge r_n$ and $\xi_n\to\xi_*\in(-1,1)$. Odd reflection across that side gives two disks in the fixed rectangle $\widetilde\Sigma=(-1,1)\times(-2,2)$. The $L^2$-normalized limiting eigenfunction is
\[
\widetilde\phi_0(x,t)=2^{-1/2}\cos\frac{\pi x}{2}\sin\frac{\pi t}{2},
\]
and $|\nabla\widetilde\phi_0(\xi_*,0)|>0$. Hence \eqref{eq:capacity-inside-lower} gives a reflected capacity bounded below by $c r_n^2$. Formula~\eqref{eq:AFHL-simple} then contradicts the assumed $O(r_n^4)$ cost.

Finally suppose $x_*=(1,1)$ after symmetry, and write
\[
x_n=(1-s_n,1-t_n),\qquad s_n,t_n\ge r_n,\qquad s_n,t_n\to0.
\]
Odd reflection across both sides gives four disks in $\widehat Q=(-1,3)^2$ with centers, relative to $(1,1)$, at $(\pm s_n,\pm t_n)$. The normalized limiting eigenfunction is
\[
\widehat\phi_0(x,y)=\frac12\cos\frac{\pi x}{2}\cos\frac{\pi y}{2},
\]
so in local coordinates
\[
\nabla\widehat\phi_0(1+X,1+Y)=\frac{\pi^2}{8}(Y,X)+O(|(X,Y)|^3).
\]
Thus $|\nabla\widehat\phi_0|^2\ge c(X^2+Y^2)$ near the corner. Since the four reflected disks have disjoint interiors,
\[
\int_{\widehat K_n}(X^2+Y^2)
=4\pi r_n^2(s_n^2+t_n^2)+2\pi r_n^4.
\]
The lower bound \eqref{eq:capacity-inside-lower} therefore yields
\[
\operatorname{Cap}_{\widehat Q}(\widehat K_n,\widehat\phi_0)
\ge c_1r_n^2(s_n^2+t_n^2+r_n^2).
\]
Together with \eqref{eq:AFHL-simple} and the $O(r_n^4)$ spectral bound, this forces $s_n^2+t_n^2=O(r_n^2)$, contrary to escape from every fixed corner layer. These cases exhaust $\overline Q$.
\end{proof}

\begin{corollary}[Two-hole localization under an $O(r^4)$ bound]\label{cor:two-hole-localization}
For every $C>0$ there exist $R>1$ and $r_0>0$ such that, if $(x_1,x_2)\in\mathcal C_r$, $0<r<r_0$, and
\[
\Lambda_r(x_1,x_2)-\frac{\pi^2}{2}\le Cr^4,
\]
then each $x_i$ belongs to one of the four sets $\mathcal N_r^R(p)$.
\end{corollary}
\begin{proof}
For $i=1,2$, domain monotonicity gives
\[
\lambda_r^{(1)}(x_i)\le\Lambda_r(x_1,x_2).
\]
Apply Proposition~\ref{prop:uniform-one-hole-localization} to each center.
\end{proof}

In particular, configurations with one obstacle centered on an open side away from the corners and the other near a corner, as well as configurations with both obstacles centered on open-side portions (including the same side), cannot be asymptotically minimizing. Indeed, the open-side obstacle alone has eigenvalue excess of order at least $r^2$, whereas an adjacent-corner competitor has cost $O(r^4)$.

\begin{remark}[Relation with the moving-plane displacement method]\label{rem:HKK-two-obstacles}
The shape-derivative/moving-plane method proposed in~\cite{HKK2001} is useful when the portion of the domain used in the reflection argument is mapped into the perforated domain after one obstacle is translated. With a second obstacle fixed, this geometric compatibility is not automatic: even for two strictly separated disks, reflection across the hyperplane through the translated disk can send points of the comparison cap into the second disk. Then the function $w(x)=u(x)-u(x')$ used in the maximum-principle argument is not defined as a solution of the same equation throughout the required reflected region, so the Hopf-lemma sign argument cannot be applied directly. Localizing the deformation vector field near the translated obstacle localizes the Hadamard boundary integral, but it does not localize this sign argument: the reflected comparison remains a global geometric statement and still sees the second obstacle. For example, take centers $c_1=(a,a)$ and $c_2=(-a,-a)$ with $r\ll a<1/3$. For the outward translation of the first disk in the $+x$ direction, reflection across $x=a$ sends the second disk to a disk centered at $(3a,-a)$, which still lies in the right comparison cap; the analogous obstruction occurs for the $+y$ translation, and by symmetry for the two outward translations of the second disk. Thus the disks need not be tangent for these natural side-parallel moving-plane comparisons to fail. Thus the method can eliminate configurations for which the required reflected inclusion holds, but tangency of the two disks is not the only possible obstruction. Proposition~\ref{prop:uniform-one-hole-localization} supplies the uniform global reduction needed here without requiring such an inclusion.
\end{remark}

\section{The side-tangent one-hole branch}\label{sec:side-branch}

This section analyzes the one-hole branch in which the disk is tangent to an open side. We first record the local side and corner profiles of the unperturbed ground state, then derive the $r^2$ expansion away from the corners, pass to the $O(r)$ endpoint scale, and finally prove that the full corner coefficient is minimized at corner tangency.

For the one-hole side-tangent branch we write
\[
\Omega_r^{\mathrm{side}}(\xi):=
Q\setminus \overline{B_r(\xi,1-r)},
\qquad
|\xi|<1-r,
\]
and
\[
\lambda_r^{\mathrm{side}}(\xi):=\lambda_1\bigl(\Omega_r^{\mathrm{side}}(\xi)\bigr).
\]

\subsection{The unperturbed square}

We begin by recording the ground state and its local expansions at an open side and at a corner; these determine the homogeneities that enter the subsequent capacity asymptotics.

\begin{lemma}\label{lem:square-ground-state-v4}
The first Dirichlet eigenpair of $Q=(-1,1)^2$ is
\[
\nu_0(x,y)=\cos\frac{\pi x}{2}\cos\frac{\pi y}{2},
\qquad
\lambda_0=\frac{\pi^2}{2}.
\]
Moreover:
\begin{enumerate}[label=\rm(\roman*)]
\item If $|\xi|<1$ and we write $x=\xi+s$, $y=1-t$, then
\[
\nu_0(\xi+s,1-t)
=
\frac{\pi}{2}\cos\frac{\pi\xi}{2}\,t
+O\!\big((|s|+t)^2\big)
\qquad\text{as }(s,t)\to(0,0).
\]
\item If we write $x=1-s$, $y=1-t$, then
\[
\nu_0(1-s,1-t)
=
\frac{\pi^2}{4}\,st
+O\!\big((s+t)^4\big)
\qquad\text{as }(s,t)\to(0,0).
\]
\end{enumerate}
\end{lemma}

\begin{proof}
The eigenpair is immediate by separation of variables.  The expansions follow from the Taylor series of $\cos\frac{\pi(\xi+s)}2\sin\frac{\pi t}{2}$ and $\sin\frac{\pi s}{2}\sin\frac{\pi t}{2}$.
\end{proof}

\subsection{Odd reflection and the away-from-corners regime}

We first consider tangency points that stay a fixed positive distance from the corners. Odd reflection across the flat side converts the boundary-touching hole into an interior compact perturbation of a simple eigenvalue.

Set $t:=1-y$.  Then the square becomes
\[
\Sigma:=(-1,1)\times(0,2),
\]
and the side-tangent hole becomes $\overline{B_r((\xi,r))}\subset\Sigma$, tangent to the flat side $t=0$.
Reflect oddly across $t=0$ and define
\[
\widetilde\Sigma:=(-1,1)\times(-2,2),
\]
\[
K_r(\xi):=\overline{B_r((\xi,r))}\cup \overline{B_r((\xi,-r))}
=(\xi,0)+rK,
\qquad
K:=\overline{B_1((0,1))}\cup \overline{B_1((0,-1))}.
\]

\begin{proposition}[Odd reflection]\label{prop:reflection-v4}
If $u\in H^1_0(\Omega_r^{\mathrm{side}}(\xi))$, its odd reflection
\[
\widetilde u(x,t)=
\begin{cases}
 u(x,t),& t>0,\\[1mm]
 -u(x,-t),& t<0,
\end{cases}
\]
belongs to $H^1_0(\widetilde\Sigma\setminus K_r(\xi))$ and is odd in $t$.  Conversely, every odd element of $H^1_0(\widetilde\Sigma\setminus K_r(\xi))$ restricts to a function in $H^1_0(\Omega_r^{\mathrm{side}}(\xi))$.  Therefore
\[
\lambda_r^{\mathrm{side}}(\xi)
=
\min_{0\neq v\in H^1_{0,\mathrm{odd}}(\widetilde\Sigma\setminus K_r(\xi))}
\frac{\int_{\widetilde\Sigma\setminus K_r(\xi)} |\nabla v|^2\,dx\,dt}
     {\int_{\widetilde\Sigma\setminus K_r(\xi)} |v|^2\,dx\,dt}.
\]
\end{proposition}

\begin{proof}
This is the standard odd-reflection argument across a flat Dirichlet boundary segment.
\end{proof}

\begin{lemma}\label{lem:reflected-eigenpair-v4}
The relevant limiting eigenpair in the reflected rectangle is
\[
\widetilde\nu_0(x,t)=\cos\frac{\pi x}{2}\sin\frac{\pi t}{2},
\qquad
\mu_0=\frac{\pi^2}{2},
\]
and $\mu_0$ is a simple eigenvalue of the full Dirichlet Laplacian on $\widetilde\Sigma$.
\end{lemma}

\begin{proof}
The Dirichlet spectrum of $\widetilde\Sigma=(-1,1)\times(-2,2)$ is
\[
\mu_{m,n}=\frac{\pi^2}{4}m^2+\frac{\pi^2}{16}n^2,
\qquad m,n\in\mathbb N.
\]
The value $\mu_0=\pi^2/2$ corresponds to $(m,n)=(1,2)$, and it is simple because $4m^2+n^2=8$ has no other solution in positive integers.
\end{proof}

For a compact set $K\subset\mathbb R^2$ and a polynomial $P$, define the planar Beppo--Levi space
\[
\dot H^1(\mathbb R^2):=\{V\in H^1_{\rm loc}(\mathbb R^2):\nabla V\in L^2(\mathbb R^2)\}/\mathbb R
\]
and the full cell capacity
\begin{equation}\label{eq:full-cell-capacity}
\mathfrak C(K,P):=\inf\left\{\int_{\mathbb R^2}|\nabla V|^2:
V\in\dot H^1(\mathbb R^2),\quad V=P\ \text{quasi-everywhere on }K\right\}.
\end{equation}
The condition on $K$ fixes the additive constant. In particular, the energy is integrated over the whole plane and contains $\int_K|\nabla P|^2$; this is the full $u$-capacity entering the eigenvalue expansion, rather than only the energy of the exterior harmonic corrector.

\begin{lemma}[Uniform capacity blow-up on compact parameter sets]\label{lem:full-capacity-blowup}
Let $D\subset\mathbb R^2$ be bounded, connected and Lipschitz, let $x_0\in D$, and let $u$ be an $L^2(D)$-normalized eigenfunction associated with a simple Dirichlet eigenvalue $\lambda_N(D)$. Suppose
\[
u(x_0+rX)=r^kP(X)+O(r^{k+1})
\]
in $C^1_{\rm loc}$, where $P$ is a nonzero homogeneous harmonic polynomial of degree $k\ge1$. Let $\Theta$ be compact and, for each $\theta\in\Theta$, let
\[
K_\theta=\bigcup_{j=1}^m \overline{B_{\rho_j(\theta)}(c_j(\theta))},
\]
where the centers $c_j$ and radii $\rho_j$ depend continuously on $\theta$, all $K_\theta$ are contained in one fixed ball, $\rho_j(\theta)\ge\rho_*>0$, and the disk interiors are pairwise disjoint. Then
\[
\operatorname{Cap}_D(x_0+rK_\theta,u)
=r^{2k}\mathfrak C(K_\theta,P)+o(r^{2k})
\]
uniformly for $\theta\in\Theta$, and consequently
\[
\lambda_N(D\setminus(x_0+rK_\theta))
=\lambda_N(D)+r^{2k}\mathfrak C(K_\theta,P)+o(r^{2k})
\]
uniformly along the eigenvalue branch converging to $\lambda_N(D)$.
\end{lemma}
\begin{proof}
After the change of variables $x=x_0+rX$, the exact two-dimensional energy scaling gives
\[
r^{-2k}\operatorname{Cap}_D(x_0+rK_\theta,u)
=\inf\int_{(D-x_0)/r}|\nabla W|^2,
\]
where $W$ agrees on $K_\theta$ with $r^{-k}u(x_0+rX)\to P$ in $C^1_{\rm loc}$. It is enough to prove convergence along an arbitrary sequence $r_n\to0$ and $\theta_n\to\theta$, since compactness of $\Theta$ then gives uniformity by contradiction. On every fixed ball, the rescaled boundary data
\[
g_n(X):=r_n^{-k}u(x_0+r_nX)
\]
converge to $P$ in $C^1$, while the disk unions $K_{\theta_n}$ converge to $K_\theta$ in the Hausdorff sense. The associated Dirichlet-energy minimization problems converge in the standard Mosco sense. For the limsup, one first truncates a whole-space admissible competitor by a logarithmic cutoff and then, on fixed tubular neighborhoods of the finitely many circular arcs, adjusts its trace from $P$ on $K_\theta$ to $g_n$ on $K_{\theta_n}$; continuity of the disk centers and radii makes the correction tend to zero in $H^1$. If two limiting disks touch, the contact set is finite and has zero planar $H^1$-capacity, so it does not change the admissible class. For the liminf, the limsup construction supplies a uniform energy bound; after fixing the additive constant by the trace on one disk, local $H^1$ compactness and lower semicontinuity give a limit whose quasi-continuous representative equals $P$ on $K_\theta$. Thus the limiting energy is at least $\mathfrak C(K_\theta,P)$. The same sequential argument with the expanding domain removed proves continuity of $\mathfrak C(K_\theta,P)$ under this disk convergence. Hence
\[
r_n^{-2k}\operatorname{Cap}_D(x_0+r_nK_{\theta_n},u)
\longrightarrow \mathfrak C(K_\theta,P),
\]
and the convergence is uniform on compact $\Theta$. The eigenvalue statement follows from~\cite[Theorem~1.4]{AFHL2019}. For a fixed shape and homogeneous profile, Theorem~6.2 of~\cite{ABLM2021} gives the same coefficient explicitly as the sum of the exterior corrector energy and the polynomial energy inside the hole.
\end{proof}

Since $\|\widetilde\nu_0\|_{L^2(\widetilde\Sigma)}=\sqrt2$, the normalized reflected eigenfunction is $2^{-1/2}\widetilde\nu_0$. Near $p_\xi=(\xi,0)$,
\[
2^{-1/2}\widetilde\nu_0(\xi+X,T)
=
\frac{\pi}{2\sqrt2}\cos\frac{\pi\xi}{2}\,T+O\bigl(|(X,T)|^2\bigr),
\]
so its first nonzero homogeneous term is
\[
P_\xi(X,T)=a_\xi T,
\qquad
 a_\xi:=\frac{\pi}{2\sqrt2}\cos\frac{\pi\xi}{2}.
\]

\begin{proposition}\label{prop:side-away-from-corners-v4}
Fix $\delta\in(0,1)$.  Then there exists a constant $\Gamma_K>0$, depending only on the reflected shape $K$, such that
\[
\lambda_r^{\mathrm{side}}(\xi)
=
\frac{\pi^2}{2}
+\Gamma_K\cos^2\frac{\pi\xi}{2}\,r^2
+o_\delta(r^2)
\qquad (r\to0),
\]
uniformly for $|\xi|\le 1-\delta$.
\end{proposition}

\begin{proof}

By Proposition~\ref{prop:reflection-v4}, $\lambda_r^{\mathrm{side}}(\xi)$ is the odd branch of the Dirichlet spectrum of $\widetilde\Sigma\setminus K_r(\xi)$. Lemma~\ref{lem:reflected-eigenpair-v4} makes the limiting eigenvalue simple in the full reflected spectrum, and the normalized local profile is
\[
P_\xi(X,T)=\frac{\pi}{2\sqrt2}\cos\frac{\pi\xi}{2}\,T.
\]
Lemma~\ref{lem:full-capacity-blowup}, equivalently the fixed-shape homogeneous expansion of~\cite[Theorem~6.2]{ABLM2021} combined with~\cite[Theorem~1.4]{AFHL2019}, therefore gives
\[
\lambda_r^{\mathrm{side}}(\xi)-\frac{\pi^2}{2}
=\frac{\pi^2}{8}\,\mathfrak C(K,T)\cos^2\frac{\pi\xi}{2}\,r^2+o_\delta(r^2)
\]
uniformly for $|\xi|\le1-\delta$. Hence the proposition holds with
\[
\Gamma_K:=\frac{\pi^2}{8}\mathfrak C(K,T)>0.
\]

\end{proof}

\begin{corollary}\label{cor:side-minimizers-approach-corners-v4}
Let $\xi_r\in(-1+r,1-r)$ minimize $\lambda_r^{\mathrm{side}}(\xi)$ over the side-tangent branch.  Then
\[
|\xi_r|\to 1
\qquad\text{as }r\to0.
\]
\end{corollary}

\begin{proof}
Assume not. Then there exist $\delta>0$ and a sequence $r_j\to0$ such that $|\xi_{r_j}|\le 1-\delta$ for all $j$.  For $j$ large, the competitor
\[
\widehat\xi_{r_j}:=\operatorname{sgn}(\xi_{r_j})(1-\delta/2)
\]
is admissible, and Proposition~\ref{prop:side-away-from-corners-v4} applies both at $\xi_{r_j}$ and at $\widehat\xi_{r_j}$.  Since
\[
\cos^2\frac{\pi(1-\delta/2)}{2}<\cos^2\frac{\pi(1-\delta)}{2},
\]
we get $\lambda_{r_j}^{\mathrm{side}}(\widehat\xi_{r_j})<\lambda_{r_j}^{\mathrm{side}}(\xi_{r_j})$ for large $j$, a contradiction.
\end{proof}

\subsection{Endpoint scaling and double reflection}

The preceding $r^2$ expansion degenerates as the tangency point approaches a corner. We therefore pass to the corner scale, where a second odd reflection produces the $r^4$ regime.

For each fixed $a\ge 1$ and all sufficiently small $r>0$, define the endpoint-scaled configuration
\[
\Omega_{r,a}:=Q\setminus \overline{B_r(1-ar,1-r)}.
\]
Thus $a=1$ corresponds to the corner-tangent hole, while $a>1$ corresponds to a side-tangent hole whose tangency point is $O(r)$ away from the corner.

Reflect oddly across $y=1$ and then oddly across $x=1$.  The outer domain becomes
\[
\widehat Q:=(-1,3)^2,
\]
and the single hole becomes the interior compact defect
\[
\widehat K_{r,a}=(1,1)+rK_a,
\]
where
\[
K_a:=
\bigcup_{\sigma_1,\sigma_2\in\{\pm1\}}
\overline{B_1\bigl((\sigma_1 a,\sigma_2)\bigr)}.
\]
The reflected eigenfunction is still
\[
\widehat\nu_0(x,y)=\cos\frac{\pi x}{2}\cos\frac{\pi y}{2},
\qquad
\lambda_0=\frac{\pi^2}{2},
\]
and this eigenvalue is simple on $\widehat Q$: indeed the Dirichlet spectrum of $(-1,3)^2$ is $(\pi^2/16)(m^2+n^2)$, and $m^2+n^2=8$ has the unique positive solution $(m,n)=(2,2)$.
Near $(1,1)$,
\[
\widehat\nu_0(1+rX,1+rY)
=
\frac{\pi^2}{4}r^2XY+O(r^4).
\]
Since $\|\widehat\nu_0\|_{L^2(\widehat Q)}=2$, the normalized reflected eigenfunction is $\widehat\nu_0/2$ and its first nonzero homogeneous term is $P(X,Y)=(\pi^2/8)XY$.

\begin{proposition}\label{prop:endpoint-asymptotic-v4}
Let $A\ge1$ be fixed. Then, uniformly for $a\in[1,A]$,

\[
\lambda_1(\Omega_{r,a})
=
\frac{\pi^2}{2}
+\frac{\pi^4}{64}\,\mathfrak M(a)\,r^4
+o_A(r^4),
\qquad
\mathfrak M(a):=\mathfrak C(K_a,XY).
\]

In particular, for each fixed $a\ge1$ the same expansion holds with $o(r^4)$.
\end{proposition}

\begin{proof}

After the double reflection, this is an interior compact-defect perturbation of the simple eigenvalue $\lambda_0$ on $\widehat Q$. The normalized reflected eigenfunction satisfies
\[
\frac12\widehat\nu_0(1+rX,1+rY)=\frac{\pi^2}{8}r^2XY+O(r^4).
\]
Lemma~\ref{lem:full-capacity-blowup} therefore gives the coefficient
$(\pi^2/8)^2\mathfrak C(K_a,XY)=\pi^4\mathfrak M(a)/64$, uniformly for $a\in[1,A]$, including the tangent endpoint $a=1$.

\end{proof}

\subsection{Reduction to the quadrant cell}

We next identify the reflected full capacity with a quarter-plane cell problem. This reduction separates the universal corner geometry from the normalization of the eigenfunction.

Set
\[
\Gamma:=\{(s,t)\in\mathbb R^2:s>0,\ t>0\},
\qquad
H_a:=\overline{B_1((a,1))},
\qquad
P_a:=\Gamma\setminus H_a.
\]
Let $W_a$ be the unique harmonic function in $P_a$ such that
\[
W_a=0 \text{ on }\partial\Gamma,
\qquad
W_a=st \text{ on }\partial H_a,
\qquad
\int_{P_a}|\nabla W_a|^2<\infty,
\]
and define
\[
\mathcal E(a):=\int_{P_a}|\nabla W_a|^2\,ds\,dt.
\]

The spectral coefficient is the corresponding \emph{full} corner capacity
\begin{equation}\label{eq:C-one-hole}
\mathcal C(a):=\mathcal E(a)+\int_{H_a}|\nabla(st)|^2\,ds\,dt
=\mathcal E(a)+\pi\Bigl(a^2+\frac32\Bigr).
\end{equation}

\begin{proposition}\label{prop:M-equals-4E-v4}

For every $a\ge1$,
\[
\mathfrak M(a)=4\mathcal C(a).
\]
Consequently,
\[
\lambda_1(\Omega_{r,a})
=
\frac{\pi^2}{2}
+\frac{\pi^4}{16}\,\mathcal C(a)\,r^4
+o(r^4).
\]

\end{proposition}

\begin{proof}

Let $\Phi_a$ be the minimizer defining $\mathfrak C(K_a,XY)$. Uniqueness and the symmetries of $K_a$ and $XY$ imply that $\Phi_a$ is odd in each coordinate. In the first quadrant it equals $W_a$ outside $H_a$ and equals $st$ on $H_a$. Therefore
\[
\mathfrak M(a)=4\left(\int_{P_a}|\nabla W_a|^2+\int_{H_a}|\nabla(st)|^2\right)=4\mathcal C(a).
\]
Substitution into Proposition~\ref{prop:endpoint-asymptotic-v4} gives the eigenvalue expansion.

\end{proof}

\subsection{Far-field coefficients and exact identities}

To prove monotonicity of the corner coefficient, we introduce the leading far-field coefficients of the relevant harmonic correctors and derive the identities that relate them to their Dirichlet energies.

Let $c(a)$ be the far-field coefficient of $W_a$:
\[
W_a(r,\theta)=c(a)r^{-2}\sin(2\theta)+O(r^{-4})
\qquad (r\to+\infty,\ 0<\theta<\tfrac{\pi}{2}).
\]
Next, let $J_a$ be the unique harmonic function in $P_a$ such that
\[
J_a=0 \text{ on }\partial\Gamma,
\qquad
J_a=t \text{ on }\partial H_a,
\qquad
\int_{P_a}|\nabla J_a|^2<\infty,
\]
and denote
\[
\mathcal J(a):=\int_{P_a}|\nabla J_a|^2\,ds\,dt,
\qquad
I(a):=\int_{P_a}\nabla W_a\cdot\nabla J_a\,ds\,dt.
\]
Let $d(a)$ be the far-field coefficient of $J_a$:
\[
J_a(r,\theta)=d(a)r^{-2}\sin(2\theta)+O(r^{-4})
\qquad (r\to+\infty,\ 0<\theta<\tfrac{\pi}{2}).
\]

\begin{proposition}\label{prop:exact-identities-v4}
For every $a\ge1$,
\begin{align}
\mathcal E(a)&=\frac{\pi}{2}c(a)-\pi\Bigl(a^2+\frac32\Bigr),\label{eq:E-from-c-v4}\\
 d(a)&=2a+\frac{2}{\pi}I(a),\label{eq:d-from-I-v4}\\
 I(a)&=a\,\mathcal J(a).\label{eq:I-equals-aJ-v4}
\end{align}
In particular, combining the first identity with \eqref{eq:C-one-hole} gives
\[
\mathcal C(a)=\frac{\pi}{2}c(a).
\]
\end{proposition}

\begin{proof}
For \eqref{eq:E-from-c-v4}, apply Green's identities to $W_a$ and $q(s,t):=st$ on $P_a\cap B_R$.  The axis contributions vanish because both $W_a$ and $q$ vanish there.  On the outer arc one uses
\[
q=\frac{r^2}{2}\sin(2\theta),
\qquad
W_a=c(a)r^{-2}\sin(2\theta)+O(r^{-4}),
\]
which yields the limit $-(\pi/2)c(a)$ for the outer boundary term.  On $\partial H_a$, since $W_a=q$, Green's identity gives
\[
\mathcal E(a)=\frac{\pi}{2}c(a)-\int_{H_a}(s^2+t^2)\,ds\,dt.
\]
A direct computation shows
\[
\int_{H_a}(s^2+t^2)\,ds\,dt=\pi\Bigl(a^2+\frac32\Bigr),
\]
and \eqref{eq:E-from-c-v4} follows.

For \eqref{eq:d-from-I-v4}, apply Green's second identity to $J_a$ and $q(s,t)=st$ on $P_a\cap B_R$.  The outer arc contributes $-(\pi/2)d(a)$, while the hole contribution is
\[
\int_{\partial H_a} st\,\partial_n J_a\,d\sigma-\int_{\partial H_a} t\,\partial_n(st)\,d\sigma
=I(a)+a\pi,
\]
because
\[
\int_{\partial H_a} t\,\partial_n(st)\,d\sigma=-\int_{H_a}s\,ds\,dt=-a\pi.
\]
Thus $0=I(a)+a\pi-(\pi/2)d(a)$, which is exactly \eqref{eq:d-from-I-v4}.

For \eqref{eq:I-equals-aJ-v4}, consider $\Psi_a:=W_a-aJ_a$.  Then $\Psi_a$ is harmonic in $P_a$, vanishes on the axes, and satisfies $\Psi_a=(s-a)t$ on $\partial H_a$.  Let
\[
\mathcal F(a):=\int_{P_a}|\nabla\Psi_a|^2.
\]
Repeating the proof of \eqref{eq:E-from-c-v4} with the harmonic polynomial $(s-a)t$ in place of $st$ yields
\[
\mathcal F(a)=\frac{\pi}{2}\bigl(c(a)-a d(a)\bigr)-\frac{3\pi}{2}.
\]
On the other hand,
\[
\mathcal F(a)=\mathcal E(a)-2aI(a)+a^2\mathcal J(a).
\]
Substituting \eqref{eq:E-from-c-v4} and \eqref{eq:d-from-I-v4} and simplifying gives $a^2\mathcal J(a)=aI(a)$, hence \eqref{eq:I-equals-aJ-v4}.
\end{proof}

\subsection{A shifted comparison and monotonicity of the endpoint coefficient}

The exact identities above turn a translated maximum-principle comparison into a monotonicity statement for the endpoint capacity.

\begin{proposition}\label{prop:shifted-comparison-v4}
If $b=a+\delta$ with $\delta>0$, then
\[
W_b(s+\delta,t)\ge W_a(s,t)+\delta J_a(s,t)
\qquad\text{for all }(s,t)\in P_a.
\]
Consequently,
\[
c(b)\ge c(a)+\delta d(a).
\]
\end{proposition}

\begin{proof}
Define $\Phi(s,t):=W_b(s+\delta,t)$ on $P_a$.  This is well-defined because
\[
(s-a)^2+(t-1)^2\ge1
\Longrightarrow
(s+\delta-b)^2+(t-1)^2\ge1.
\]
Thus $\Phi$ is harmonic in $P_a$.  On the circular boundary $\partial H_a$,
\[
\Phi(s,t)=W_b(s+\delta,t)=(s+\delta)t=st+\delta t=W_a(s,t)+\delta J_a(s,t).
\]
On the bottom side $t=0$, all three functions vanish.  On the vertical side $s=0$, one has
\[
\Phi(0,t)=W_b(\delta,t)\ge0=W_a(0,t)+\delta J_a(0,t).
\]
Hence $\Phi-(W_a+\delta J_a)$ is harmonic in $P_a$ and nonnegative on the whole boundary.  The maximum principle on truncated domains and passage to the limit imply the claimed inequality in $P_a$.

For the far-field coefficient, evaluate along any fixed interior ray.  Since translation by $\delta$ changes only lower-order terms at infinity,
\[
W_b(s+\delta,t)=c(b)r^{-2}\sin(2\theta)+O(r^{-3}).
\]
Comparing with the far-field expansions of $W_a$ and $J_a$ gives
\[
c(b)\ge c(a)+\delta d(a).
\]
\end{proof}

\begin{proposition}\label{prop:J-lower-bound-v4}
For every $a\ge1$,
\[
\mathcal J(a)\ge \mathcal J_\infty,
\qquad
\mathcal J_\infty=\pi\Bigl(\frac{\pi^2}{3}-1\Bigr) > \pi,
\]
where $\mathcal J_\infty$ is the half-plane comparison energy defined below.
\end{proposition}

\begin{proof}
Translate the hole to the fixed position $(0,1)$ by setting $x=s-a$, $y=t$.  Then
\[
\Omega_a:=\{x>-a,\ y>0\}\setminus \overline{B_1((0,1))},
\]
and $\mathcal J(a)$ is the minimum of the Dirichlet energy among functions equal to $0$ on $y=0$ and on $x=-a$, and equal to $y$ on $\partial B_1((0,1))$.
Let
\[
\Omega_\infty:=\{y>0\}\setminus \overline{B_1((0,1))}
\]
and denote by $\mathcal J_\infty$ the analogous energy when the vertical boundary $x=-a$ is removed.  If $u$ is admissible on $\Omega_a$, extend it by $0$ to $\{x<-a,\ y>0\}$.  The extension is admissible on $\Omega_\infty$ and has the same energy, hence $\mathcal J(a)\ge\mathcal J_\infty$.

It remains to compute $\mathcal J_\infty$.  Oddly reflect across $y=0$.  This doubles the energy and gives a harmonic function on the complement of the two tangent disks
\[
\overline{B_1((0,1))}\cup\overline{B_1((0,-1))}
\]
with boundary data $Y$ on both circles.  Invert at the tangency point:
\[
w=\frac1z,
\qquad z=X+iY,
\qquad w=u+iv.
\]
This maps the exterior of the two tangent disks conformally onto the strip
\[
S:=\Bigl\{-\frac12<v<\frac12\Bigr\}.
\]
On the boundary lines $v=\pm\tfrac12$, the transformed boundary data are
\[
\pm f(u),
\qquad
f(u)=\frac{1}{2(u^2+\frac14)}.
\]
Let $U$ denote the harmonic solution in $S$ with these boundary values.  In two dimensions the Dirichlet integral $\int|\nabla U|^2$ is invariant under conformal changes of variables; this energy, not the eigenvalue itself, is the conformally invariant quantity used here. Hence its strip energy equals $2\mathcal J_\infty$.

Taking Fourier transform in $u$, with convention
\[
\widehat f(k):=\int_{\mathbb R} f(u)e^{-iku}\,du,
\]
one finds
\[
\widehat f(k)=\pi e^{-|k|/2},
\]
and
\[
\widehat U(k,v)=\widehat f(k)\,\frac{\sinh(|k|v)}{\sinh(|k|/2)}.
\]
A mode-by-mode energy computation yields
\[
2\mathcal J_\infty
=
\frac{1}{\pi}\int_{\mathbb R}|k|\coth\!\Bigl(\frac{|k|}{2}\Bigr)|\widehat f(k)|^2\,dk
=
2\pi\int_0^\infty k\coth\!\Bigl(\frac{k}{2}\Bigr)e^{-k}\,dk.
\]
Using
\[
\coth\!\Bigl(\frac{k}{2}\Bigr)e^{-k}=e^{-k}+2\sum_{m\ge2}e^{-mk},
\]
we obtain
\[
\int_0^\infty k\coth\!\Bigl(\frac{k}{2}\Bigr)e^{-k}\,dk
=1+2\sum_{m\ge2}\frac1{m^2}=\frac{\pi^2}{3}-1.
\]
Therefore
\[
\mathcal J_\infty=\pi\Bigl(\frac{\pi^2}{3}-1\Bigr)>\pi,
\]
as claimed.
\end{proof}

\begin{theorem}\label{thm:endpoint-coefficient-monotone-v4}
The endpoint coefficient $\mathcal E(a)$ is strictly increasing on $[1,\infty)$.
In particular, for every fixed $a>1$,
\[
\mathcal E(a)>\mathcal E(1).
\]
\end{theorem}

\begin{proof}
From \eqref{eq:I-equals-aJ-v4} and Proposition~\ref{prop:J-lower-bound-v4},
\[
I(a)=a\mathcal J(a)\ge a\pi\Bigl(\frac{\pi^2}{3}-1\Bigr).
\]
Substituting into \eqref{eq:d-from-I-v4} gives
\[
d(a)\ge 2a+2a\Bigl(\frac{\pi^2}{3}-1\Bigr)=\frac{2\pi^2}{3}a.
\]
Now let $1\le a<b$ and choose any partition
\[
a=x_0<x_1<\cdots<x_n=b.
\]
Applying Proposition~\ref{prop:shifted-comparison-v4} on each subinterval,
\[
c(x_i)-c(x_{i-1})\ge (x_i-x_{i-1})d(x_{i-1})
\ge \frac{2\pi^2}{3}x_{i-1}(x_i-x_{i-1}).
\]
Summing yields
\[
c(b)-c(a)\ge \frac{2\pi^2}{3}\sum_{i=1}^n x_{i-1}(x_i-x_{i-1}).
\]
Taking the supremum over all partitions and using the fact that the left Riemann sums for the increasing function $x\mapsto x$ converge to its integral, we obtain
\[
c(b)-c(a)\ge \frac{2\pi^2}{3}\int_a^b x\,dx=\frac{\pi^2}{3}(b^2-a^2).
\]
Finally, by \eqref{eq:E-from-c-v4},
\[
\mathcal E(b)-\mathcal E(a)
=\frac{\pi}{2}\bigl(c(b)-c(a)\bigr)-\pi(b^2-a^2)
\ge \pi\Bigl(\frac{\pi^2}{6}-1\Bigr)(b^2-a^2)>0,
\]
because $\pi^2/6>1$.  This proves the theorem.
\end{proof}

\begin{corollary}\label{cor:C-endpoint-monotone}
The full capacity $\mathcal C(a)$ is strictly increasing on $[1,\infty)$.
\end{corollary}
\begin{proof}
This follows from Theorem~\ref{thm:endpoint-coefficient-monotone-v4} and
$\mathcal C(a)=\mathcal E(a)+\pi(a^2+3/2)$. Equivalently, it follows directly from $\mathcal C(a)=(\pi/2)c(a)$ and Proposition~\ref{prop:shifted-comparison-v4}, since $d(a)>0$.
\end{proof}

\begin{corollary}\label{cor:endpoint-family-corner-best-v4}
For the endpoint-scaled family $\Omega_{r,a}$ of Proposition~\ref{prop:endpoint-asymptotic-v4}, the leading coefficient in the eigenvalue expansion is uniquely minimized at $a=1$:

\[
\lambda_1(\Omega_{r,a})
=
\frac{\pi^2}{2}
+\frac{\pi^4}{16}\,\mathcal C(a)\,r^4
+o(r^4),
\qquad
\mathcal C(a)>\mathcal C(1)\quad(a>1).
\]

Thus, within the exact endpoint scaling family, the corner-tangent hole is asymptotically optimal.
\end{corollary}

\begin{proposition}[Intermediate matching regime]\label{prop:intermediate-matching-v5}
Let
\[
\xi_r=1-a_r r,
\qquad
a_r\to+\infty,
\qquad
a_r r\to0,
\]
and consider the side-tangent one-hole domains
\[
\Omega_r^{\mathrm{side}}(\xi_r)=Q\setminus \overline{B_r(\xi_r,1-r)}.
\]
Then
\[
\lambda_1\bigl(\Omega_r^{\mathrm{side}}(\xi_r)\bigr)
=
\frac{\pi^2}{2}
+\Gamma_K\cos^2\frac{\pi\xi_r}{2}\,r^2
+o\!\Bigl(\cos^2\frac{\pi\xi_r}{2}\,r^2\Bigr),
\]
and hence
\[
\lambda_1\bigl(\Omega_r^{\mathrm{side}}(\xi_r)\bigr)
=
\frac{\pi^2}{2}
+\Gamma_K\frac{\pi^2}{4}\,a_r^2r^4
+o(a_r^2r^4).
\]
In particular,
\[
\frac{\lambda_1(\Omega_r^{\mathrm{side}}(\xi_r))-\pi^2/2}{r^4}\to+\infty.
\]
\end{proposition}

\begin{proof}
Reflect across the top side exactly as in Proposition~\ref{prop:reflection-v4}. The reflected defect is
\[
K_r(\xi_r)=\overline{B_r((\xi_r,r))}\cup \overline{B_r((\xi_r,-r))}
\subset \widetilde\Sigma.
\]
Let
\[
p_r:=(\xi_r,0),
\qquad
d_r:=1-\xi_r=a_r r.
\]
Since $a_r\to+\infty$, we have $d_r/r\to+\infty$. Choose the mesoscopic radius
\[
\rho_r:=\sqrt{rd_r}=r\sqrt{a_r}.
\]
Then
\[
r\ll \rho_r\ll d_r,
\]
so $B_{\rho_r}(p_r)\subset \widetilde\Sigma$ for $r$ small, and on the defect scale $r$ the vertical sides of the rectangle escape to infinity.

Write the reflected unperturbed eigenfunction as
\[
\widetilde\nu_0(x,t)=\cos\frac{\pi x}{2}\sin\frac{\pi t}{2}.
\]
Setting
\[
\alpha_r:=\frac{\pi}{2}\cos\frac{\pi\xi_r}{2},
\]
its expansion around $p_r$ is
\[
\widetilde\nu_0(p_r+rX,rT)=\alpha_r\,rT+O\!\bigl(r^2(|X|+|T|)^2\bigr).
\]
On the mesoscopic region $|(X,T)|\le \rho_r/r=\sqrt{a_r}$, the remainder is $O(r^2a_r)=O(rd_r)$, while the linear term has size $O(\alpha_r r\sqrt{a_r})$. Since
\[
\alpha_r\sim \frac{\pi^2}{4}d_r,
\]
the relative error is $O(a_r^{-1/2})\to0$. Thus the local profile seen by the defect is asymptotically the linear function $T$ with amplitude $\alpha_r$.

Applying the same two-sided blow-up argument as in Proposition~\ref{prop:side-away-from-corners-v4}, now with the moving centers $p_r$ and the mesoscopic truncation above, gives
\[
\operatorname{Cap}_{\widetilde\Sigma}\!\bigl(K_r(\xi_r),\widetilde\nu_0\bigr)
=
M_K\,\alpha_r^2\,r^2\,(1+o(1)).
\]
Since the reflected eigenvalue remains simple, the same variational/capacitary reduction as in~\cite[Theorem~1.4]{AFHL2019} yields
\[
\lambda_1\bigl(\Omega_r^{\mathrm{side}}(\xi_r)\bigr)-\frac{\pi^2}{2}
=
\operatorname{Cap}_{\widetilde\Sigma}\!\bigl(K_r(\xi_r),\widetilde\nu_0\bigr)\,(1+o(1)).
\]
Hence
\[
\lambda_1\bigl(\Omega_r^{\mathrm{side}}(\xi_r)\bigr)
=
\frac{\pi^2}{2}
+\Gamma_K\cos^2\frac{\pi\xi_r}{2}\,r^2
+o\!\Bigl(\cos^2\frac{\pi\xi_r}{2}\,r^2\Bigr).
\]
Finally,
\[
\cos\frac{\pi\xi_r}{2}
=
\sin\frac{\pi a_r r}{2}
\sim \frac{\pi}{2}a_r r,
\]
which gives the stated $a_r^2r^4$ law and the divergence after division by $r^4$.
\end{proof}

\begin{theorem}[Asymptotic classification of the one-hole side branch]\label{thm:one-hole-side-branch-v5}
Let $\xi_r\in[-1+r,1-r]$ minimize $\lambda_r^{\mathrm{side}}(\xi)$ over the side-tangent one-hole branch. Then
\[
\frac{1-|\xi_r|}{r}\to 1
\qquad\text{as }r\to0.
\]
Equivalently, among side-tangent one-hole configurations, the asymptotically optimal one is the corner-tangent hole.
\end{theorem}

\begin{proof}
By Corollary~\ref{cor:side-minimizers-approach-corners-v4}, every minimizing sequence satisfies $|\xi_r|\to1$. Passing to a subsequence and using symmetry, write
\[
\xi_r=1-a_rr,
\qquad a_r\ge1.
\]

The corner-tangent competitor $a=1$ and Proposition~\ref{prop:M-equals-4E-v4} give
\[
\lambda_r^{\mathrm{side}}(\xi_r)-\frac{\pi^2}{2}
\le \lambda_1(\Omega_{r,1})-\frac{\pi^2}{2}=O(r^4).
\]
Proposition~\ref{prop:uniform-one-hole-localization}, applied to the center $(\xi_r,1-r)$, therefore shows that $a_r$ is bounded. Let $a_*$ be any subsequential limit. If $a_*>1$, choose $a_0\in(1,a_*)$ and $A>a_*$ so that $a_r\in[a_0,A]$ along a further subsequence. Uniformity on compact parameter sets in Proposition~\ref{prop:endpoint-asymptotic-v4} and strict monotonicity of $\mathcal C$ from Corollary~\ref{cor:C-endpoint-monotone} yield
\[
\lambda_1(\Omega_{r,a_r})-\lambda_1(\Omega_{r,1})
\ge \frac{\pi^4}{16}\bigl(\mathcal C(a_0)-\mathcal C(1)\bigr)r^4+o_A(r^4)>0
\]
for $r$ small, contradicting minimality. Hence every subsequential limit equals $1$, so $a_r\to1$.

\end{proof}

\section{Comparison of adjacent and opposite corner pairs}\label{sec:adjacent-vs-opposite}

This section establishes the exact two-corner comparison needed later in the proof. We compare the two genuinely distinct corner-tangent configurations.
For $0<r<\tfrac12$, define
\[
\Omega_r^{\mathrm{opp}}
:=
Q\setminus\Bigl(
\overline{B_r(1-r,1-r)}\cup \overline{B_r(-1+r,-1+r)}
\Bigr),
\]
\[
\Omega_r^{\mathrm{adj}}
:=
Q\setminus\Bigl(
\overline{B_r(1-r,1-r)}\cup \overline{B_r(-1+r,1-r)}
\Bigr).
\]
Thus $\Omega_r^{\mathrm{opp}}$ corresponds to the opposite-corner placement, while
$\Omega_r^{\mathrm{adj}}$ corresponds to the adjacent-corner placement along the top side.
The next theorem shows that the adjacent configuration is strictly better throughout a fixed small-radius interval, not merely asymptotically.

\begin{theorem}[Adjacent corners beat opposite corners]\label{thm:adjacent-better-than-opposite-v6}
For every $0<r<\tfrac14$,
\[
\lambda_1\bigl(\Omega_r^{\mathrm{adj}}\bigr)
<
\lambda_1\bigl(\Omega_r^{\mathrm{opp}}\bigr).
\]
\end{theorem}

\begin{proof}
Let
\[
\sigma(x,y):=(x,-y),
\qquad
H_-:=\{(x,y)\in\mathbb R^2:y<0\},
\qquad
H_+:=\{(x,y)\in\mathbb R^2:y>0\}.
\]
We use the polarization with respect to the horizontal axis that places the larger value on the lower half-plane.
Namely, if $u\ge 0$ is measurable on $Q$, extended by $0$ outside its support, define
\[
(Pu)(z):=
\begin{cases}
\max\{u(z),u(\sigma z)\}, & z\in H_-,\\[1mm]
\min\{u(z),u(\sigma z)\}, & z\in H_+,\\[1mm]
u(z), & z\in \{y=0\}.
\end{cases}
\]
Standard properties of polarization imply that if $u\in H_0^1(\Omega)$, then
$Pu\in H_0^1(P\Omega)$, where $P\Omega$ is the polarized support of $u$, and moreover
\[
\int_Q |Pu|^2\,dx\,dy = \int_Q |u|^2\,dx\,dy,
\qquad
\int_Q |\nabla Pu|^2\,dx\,dy = \int_Q |\nabla u|^2\,dx\,dy;
\]
see Brock--Solynin~\cite{BrockSolynin2000}. A direct inspection of the two reflected corner holes shows that
\[
P\Omega_r^{\mathrm{opp}}=\Omega_r^{\mathrm{adj}}.
\]
Indeed, the top-right hole is fixed, while the bottom-left hole is moved to the top-left corner.

A corner-tangent disk may cut off a small corner pocket, so we first identify the component carrying the first eigenvalue. Each such pocket is contained in a square of side $r$, and hence has first Dirichlet eigenvalue at least $2\pi^2/r^2$. For $r<1/4$, the large component contains $(-1/2,1/2)^2$, so its first eigenvalue is at most $2\pi^2$. Thus the first eigenvalue of the full punctured square is the simple first eigenvalue of its large component.

Let $u_{\mathrm{opp}}$ be the positive $L^2$-normalized first eigenfunction of that large component, extended by $0$ to any corner pockets and then to $Q$.
Set
\[
v:=Pu_{\mathrm{opp}}.
\]
Then $v\in H_0^1(\Omega_r^{\mathrm{adj}})$ and, by the equalities above,
\[
\frac{\int_{\Omega_r^{\mathrm{adj}}}|\nabla v|^2}{\int_{\Omega_r^{\mathrm{adj}}}|v|^2}
=
\frac{\int_{\Omega_r^{\mathrm{opp}}}|\nabla u_{\mathrm{opp}}|^2}{\int_{\Omega_r^{\mathrm{opp}}}|u_{\mathrm{opp}}|^2}
=
\lambda_1\bigl(\Omega_r^{\mathrm{opp}}\bigr).
\]
By the Rayleigh characterization of the first Dirichlet eigenvalue,
\[
\lambda_1\bigl(\Omega_r^{\mathrm{adj}}\bigr)
\le
\lambda_1\bigl(\Omega_r^{\mathrm{opp}}\bigr).
\]

It remains to prove that equality cannot hold. Assume for contradiction that
\[
\lambda_1\bigl(\Omega_r^{\mathrm{adj}}\bigr)
=
\lambda_1\bigl(\Omega_r^{\mathrm{opp}}\bigr).
\]
Then $v$ also minimizes the Rayleigh quotient on $\Omega_r^{\mathrm{adj}}$, hence $v$ is a first Dirichlet eigenfunction there. Since the first eigenvalue is simple and the coefficients are analytic, $v$ is real-analytic in the interior of $\Omega_r^{\mathrm{adj}}$.

Consider the open set
\[
U:=\{(x,y)\in H_+ : (x,y)\in \Omega_r^{\mathrm{opp}} \text{ and } (x,-y)\in \Omega_r^{\mathrm{opp}}\}.
\]

The set $U$ may contain small components behind corner-tangent disks. Let $U_0$ be the component containing the strip
\[
(-1,1)\times(0,1-2r).
\]

On $U_0$, define
\[
D(x,y):=u_{\mathrm{opp}}(x,y)-u_{\mathrm{opp}}(x,-y).
\]
Since both branches solve
\[
-\Delta u = \lambda_1\bigl(\Omega_r^{\mathrm{opp}}\bigr)u
\]
in neighborhoods of points of $U_0$, the function $D$ is real-analytic on $U_0$.
Moreover, $D$ changes sign in $U_0$.
Indeed, near the inward-facing lower arc of the top-right hole one has $u_{\mathrm{opp}}(x,y)$ arbitrarily small while $u_{\mathrm{opp}}(x,-y)>0$, so $D<0$ nearby.
On the other hand, near the inward-facing lower arc of the reflected image of the bottom-left hole in the top-left quadrant, one has $u_{\mathrm{opp}}(x,y)>0$ while $u_{\mathrm{opp}}(x,-y)$ is arbitrarily small, so $D>0$ nearby.

Both neighborhoods connect vertically to the displayed strip and hence belong to $U_0$. Therefore $D$ takes both signs in the connected open set $U_0$, so its interior nodal set is nonempty and separates positive and negative regions. The planar nodal-set regularity theorem for solutions of $-\Delta D=\lambda D$~\cite{Cheng1976} states that the nodal set is smooth away from a lower-dimensional singular set; hence some separating nodal point is regular. Thus there exists $z_0\in U_0$ such that
\[
D(z_0)=0,
\qquad
\nabla D(z_0)\neq 0.
\]

Now, on $U_0\subset H_+$ we have
\[
v(x,y)=\min\{u_{\mathrm{opp}}(x,y),u_{\mathrm{opp}}(x,-y)\}.
\]
Thus on one side of the regular nodal arc through $z_0$, the function $v$ equals $u_{\mathrm{opp}}(x,y)$, while on the other side it equals $u_{\mathrm{opp}}(x,-y)$.
At $z_0$ these two analytic branches have the same value, but their gradients differ by
\[
\nabla D(z_0)\neq 0.
\]
Hence $v$ fails to be $C^1$ at the interior point $z_0$.
This contradicts the interior regularity of an eigenfunction.
Therefore equality is impossible, and the inequality is strict.
\end{proof}

\begin{corollary}[Opposite corner layers polarize to adjacent layers]\label{cor:opposite-layers-polarize}
Fix $R>1$. For all sufficiently small $r$, let
\[
x_+\in\mathcal N_r^R((1,1)),\qquad
x_-\in\mathcal N_r^R((-1,-1)),\qquad
x_-^*:=(x_{-,1},-x_{-,2})\in\mathcal N_r^R((-1,1)).
\]
Then
\[
\lambda_1\!\left(Q\setminus\bigl(\overline{B_r(x_+)}\cup\overline{B_r(x_-^*)}\bigr)\right)
<
\lambda_1\!\left(Q\setminus\bigl(\overline{B_r(x_+)}\cup\overline{B_r(x_-)}\bigr)\right).
\]
The analogous statement holds for every diagonally opposite pair by symmetry.
\end{corollary}
\begin{proof}
Write
\[
x_+=(1-a_+r,1-b_+r),\qquad x_-=(-1+a_-r,-1+b_-r),\qquad (a_\pm,b_\pm)\in[1,R]^2.
\]
For $r$ sufficiently small, the two corner boxes are disjoint. Every component cut off behind a disk is contained in a square of side $(R+2)r$, so its first eigenvalue is bounded below by $c_Rr^{-2}$, whereas the component containing a fixed central square has uniformly bounded first eigenvalue. Thus the first eigenfunction is the simple positive eigenfunction of the large component, extended by zero to any corner pockets. Applying the horizontal polarization from Theorem~\ref{thm:adjacent-better-than-opposite-v6} fixes the upper-right disk and reflects the lower-left disk to $x_-^*$, giving the non-strict inequality. The same connected-comparison and regular-nodal-crossing argument as in that theorem makes the inequality strict.
\end{proof}

\section{Analysis of configurations concentrated near one corner}\label{sec:same-corner}

We next treat the same-corner branch, namely configurations in which both holes lie in the same $O(r)$ corner layer. The goal of this section is to prove Theorem~\ref{thm:same-corner-branch-complete-v11}, which excludes this entire branch from the final minimization problem. The corresponding quadrant cell is defined as follows.
\[
\Gamma:=\{(s,t)\in\mathbb R^2:s>0,\ t>0\},
\qquad
c_i=(a_i,b_i),\qquad a_i,b_i\ge 1,
\qquad
|c_1-c_2|\ge 2,
\]
\[
H_i:=\overline{B_1(c_i)}\quad (i=1,2),
\qquad
P(c_1,c_2):=\Gamma\setminus (H_1\cup H_2).
\]
Let $W_{c_1,c_2}$ be the unique harmonic function in $P(c_1,c_2)$ such that
\[
W_{c_1,c_2}=0 \text{ on }\partial\Gamma,
\qquad
W_{c_1,c_2}=st \text{ on }\partial H_1\cup\partial H_2,
\qquad
\int_{P(c_1,c_2)}|\nabla W_{c_1,c_2}|^2<\infty,
\]
and define
\[
\mathcal E_2(c_1,c_2):=
\int_{P(c_1,c_2)}|\nabla W_{c_1,c_2}|^2\,ds\,dt.
\]

The full two-hole cell capacity is
\[
\mathcal C_2(c_1,c_2):=\mathcal E_2(c_1,c_2)
+\sum_{i=1}^2\int_{H_i}|\nabla(st)|^2\,ds\,dt
=\mathcal E_2(c_1,c_2)+\pi\sum_{i=1}^2\left(a_i^2+b_i^2+\frac12\right).
\]
The quantity entering the eigenvalue expansion is $\mathcal C_2$, not $\mathcal E_2$ alone. Accordingly, the goal in this section is to prove
\[
\mathcal C_2(c_1,c_2)>2\mathcal C(1)
\]
for every admissible distinct pair $(c_1,c_2)$.

The argument ultimately rests on the monotone one-hole family $\mathcal F$, a reflected comparison principle, and a scalar threshold inequality. We also record a coercive estimate for $\mathcal E_2$, since it gives a useful global bound on the same-corner cell. For related analyses of small holes near flat boundaries and conical corners, see~\cite{BDDM2016-2D,CDDM2017,CDDM2025}.

\subsection{A coercive estimate}

Before the sharper reflected comparison, we record elementary slice estimates that control configurations whose scaled centers move far from the corner. They reduce the later exclusion argument to bounded corner-scale parameters.

For $i=1,2$, define
\[
h_i(t):=\sqrt{1-(t-b_i)^2}
\qquad (|t-b_i|\le 1),
\]
so that the horizontal section of $H_i$ at height $t$ is
\[
[a_i-h_i(t),\,a_i+h_i(t)].
\]
Likewise set
\[
k_i(s):=\sqrt{1-(s-a_i)^2}
\qquad (|s-a_i|\le 1),
\]
so that the vertical section of $H_i$ at abscissa $s$ is
\[
[b_i-k_i(s),\,b_i+k_i(s)].
\]

\begin{lemma}[Slice lower bounds]\label{lem:same-corner-slice-v7}
Let $(c_1,c_2)$ be admissible.
\begin{enumerate}[label=\rm(\roman*)]
\item Suppose $J\subset [b_i-1,b_i+1]$ is an interval such that
\[
J\cap [b_j-1,b_j+1]=\varnothing
\qquad (j\neq i).
\]
Then
\[
\mathcal E_2(c_1,c_2)
\ge
\int_J t^2\bigl(a_i-h_i(t)\bigr)\,dt.
\]
\item Suppose $K\subset [a_i-1,a_i+1]$ is an interval such that
\[
K\cap [a_j-1,a_j+1]=\varnothing
\qquad (j\neq i).
\]
Then
\[
\mathcal E_2(c_1,c_2)
\ge
\int_K s^2\bigl(b_i-k_i(s)\bigr)\,ds.
\]
\end{enumerate}
\end{lemma}

\begin{proof}
We prove \rm(i); the proof of \rm(ii) is identical after interchanging $s$ and $t$.
Fix $t\in J$.
Because $J$ does not meet the $t$-projection of the other hole, the horizontal segment
\[
(0,\,a_i-h_i(t))\times\{t\}
\]
lies in $P(c_1,c_2)$.
Moreover,
\[
W_{c_1,c_2}(0,t)=0,
\qquad
W_{c_1,c_2}(a_i-h_i(t),t)=t\bigl(a_i-h_i(t)\bigr),
\]
since $(a_i-h_i(t),t)\in \partial H_i$ and the boundary datum on $\partial H_i$ is $st$.
Therefore, by Cauchy--Schwarz on the interval $(0,a_i-h_i(t))$,
\[
\int_0^{a_i-h_i(t)} |\partial_s W_{c_1,c_2}(s,t)|^2\,ds
\ge
\frac{|W_{c_1,c_2}(a_i-h_i(t),t)-W_{c_1,c_2}(0,t)|^2}{a_i-h_i(t)}
=
t^2\bigl(a_i-h_i(t)\bigr).
\]
Integrating over $t\in J$ and using
\[
\mathcal E_2(c_1,c_2)
\ge
\int_J\int_0^{a_i-h_i(t)} |\partial_s W_{c_1,c_2}(s,t)|^2\,ds\,dt
\]
gives the claim.
\end{proof}

\begin{lemma}[Exclusive intervals of uniform length]\label{lem:exclusive-interval-v7}
Let $(c_1,c_2)$ be admissible.
Then at least one of the following alternatives holds:
\begin{enumerate}[label=\rm(\alph*)]
\item $|b_1-b_2|\ge \sqrt2$, and for each $i\in\{1,2\}$ the interval $[b_i-1,b_i+1]$ contains a subinterval of length $\sqrt2$ disjoint from $[b_j-1,b_j+1]$;
\item $|a_1-a_2|\ge \sqrt2$, and for each $i\in\{1,2\}$ the interval $[a_i-1,a_i+1]$ contains a subinterval of length $\sqrt2$ disjoint from $[a_j-1,a_j+1]$.
\end{enumerate}
\end{lemma}

\begin{proof}
Since $|c_1-c_2|\ge 2$, one has
\[
|a_1-a_2|^2+|b_1-b_2|^2\ge 4,
\]
so at least one of $|a_1-a_2|$, $|b_1-b_2|$ is at least $\sqrt2$.
This proves the dichotomy.
Assume, for instance, that $|b_1-b_2|\ge \sqrt2$ and, after relabelling if necessary, that $b_1\le b_2$.
If $b_2-b_1\ge 2$, then the intervals $[b_1-1,b_1+1]$ and $[b_2-1,b_2+1]$ are disjoint, so the claim is immediate.
If instead $\sqrt2\le b_2-b_1\le 2$, then
\[
[b_1-1,b_2-1]\subset [b_1-1,b_1+1]\setminus [b_2-1,b_2+1],
\]
and
\[
[b_1+1,b_2+1]\subset [b_2-1,b_2+1]\setminus [b_1-1,b_1+1].
\]
Both intervals have length $b_2-b_1\ge \sqrt2$.
Hence each $[b_i-1,b_i+1]$ contains a subinterval of length $\sqrt2$ disjoint from the other $t$-interval.
The $a$-alternative is identical.
\end{proof}

Set
\[
\ell_0:=\sqrt2,
\qquad
A_0:=\int_0^{\ell_0} t^2\,dt=\frac{2\sqrt2}{3},
\]
and
\[
B_0:=\inf\Bigl\{
\int_I \bigl(1-\sqrt{1-x^2}\bigr)\,dx:
I\subset[-1,1]\text{ interval of length }\ell_0
\Bigr\}.
\]
Since $x\mapsto 1-\sqrt{1-x^2}$ is even and strictly increasing in $|x|$, the minimum is attained at the centered interval $[-1/\sqrt2,1/\sqrt2]$.
Therefore
\[
B_0
=
2\int_0^{1/\sqrt2}\bigl(1-\sqrt{1-x^2}\bigr)\,dx
=
\sqrt2-\frac12-\frac{\pi}{4}
>0.
\]

\begin{proposition}[Coercive lower bounds for the same-corner cell]\label{prop:same-corner-coercive-v7}
Let $(c_1,c_2)$ be admissible.
\begin{enumerate}[label=\rm(\roman*)]
\item If $|b_1-b_2|\ge \sqrt2$, then for each $i\in\{1,2\}$,
\[
\mathcal E_2(c_1,c_2)
\ge
A_0(a_i-1)+B_0(b_i-1)^2.
\]
\item If $|a_1-a_2|\ge \sqrt2$, then for each $i\in\{1,2\}$,
\[
\mathcal E_2(c_1,c_2)
\ge
A_0(b_i-1)+B_0(a_i-1)^2.
\]
\end{enumerate}
Consequently,
\[
\mathcal E_2(c_1,c_2)\to+\infty
\qquad\text{whenever}\qquad
\max\{a_1,b_1,a_2,b_2\}\to+\infty.
\]
\end{proposition}

\begin{proof}
Assume first that $|b_1-b_2|\ge \sqrt2$.
By Lemma~\ref{lem:exclusive-interval-v7}, for each $i\in\{1,2\}$ there exists an interval $J_i\subset [b_i-1,b_i+1]$ of length $\ell_0$ disjoint from $[b_j-1,b_j+1]$.
Applying Lemma~\ref{lem:same-corner-slice-v7}(i), we get
\[
\mathcal E_2(c_1,c_2)
\ge
\int_{J_i} t^2\bigl(a_i-h_i(t)\bigr)\,dt.
\]
Since
\[
a_i-h_i(t)=(a_i-1)+\bigl(1-h_i(t)\bigr),
\]
it follows that
\[
\mathcal E_2(c_1,c_2)
\ge
(a_i-1)\int_{J_i} t^2\,dt
+
\int_{J_i} t^2\bigl(1-h_i(t)\bigr)\,dt.
\]
Because $b_i\ge 1$ and $J_i\subset [b_i-1,b_i+1]\subset [0,\infty)$, the monotonicity of $t\mapsto t^2$ on $[0,\infty)$ gives
\[
\int_{J_i} t^2\,dt\ge A_0.
\]
Also, $t^2\ge (b_i-1)^2$ on $J_i$, so
\[
\int_{J_i} t^2\bigl(1-h_i(t)\bigr)\,dt
\ge
(b_i-1)^2\int_{J_i}\bigl(1-h_i(t)\bigr)\,dt.
\]
After the change of variable $x=t-b_i$, the image $I_i:=J_i-b_i$ is an interval of length $\ell_0$ contained in $[-1,1]$, and
\[
\int_{J_i}\bigl(1-h_i(t)\bigr)\,dt
=
\int_{I_i}\bigl(1-\sqrt{1-x^2}\bigr)\,dx
\ge B_0.
\]
Hence
\[
\mathcal E_2(c_1,c_2)
\ge
A_0(a_i-1)+B_0(b_i-1)^2.
\]
This proves \rm(i).
Statement \rm(ii) follows in the same way from Lemma~\ref{lem:exclusive-interval-v7}(b) and Lemma~\ref{lem:same-corner-slice-v7}(ii).
The coercivity conclusion is immediate.
\end{proof}

\begin{corollary}[Compact reduction for the same-corner exclusion]\label{cor:same-corner-compact-v7}
There exists $R_*>1$ such that, for every admissible pair $(c_1,c_2)$,
\[
\max\{a_1,b_1,a_2,b_2\}\ge R_*
\qquad\Longrightarrow\qquad
\mathcal E_2(c_1,c_2)>2\mathcal E(1).
\]
Consequently, the auxiliary exterior-energy inequality
\[
\mathcal E_2(c_1,c_2)>2\mathcal E(1)
\]
needs to be verified only on the compact admissible region
\[
\mathcal K_{R_*}:=
\Bigl\{(c_1,c_2):
 a_i,b_i\in[1,R_*]\ (i=1,2),\ |c_1-c_2|\ge 2
\Bigr\}.
\]
\end{corollary}

\begin{proof}
This follows immediately from Proposition~\ref{prop:same-corner-coercive-v7}, since $2\mathcal E(1)$ is fixed and finite.
\end{proof}

We now turn to the main comparison argument based on the one-hole family and the reflected comparison principle.

\subsection{A monotone one-hole corner-cell family}

We now introduce the one-obstacle quadrant problem used to compare a two-obstacle configuration with either of its constituent obstacles. The main point is coordinatewise monotonicity of its full capacity.

For $a,b\ge 1$, let
\[
H_{a,b}:=\overline{B_1((a,b))},
\qquad
P_{a,b}:=\Gamma\setminus H_{a,b},
\]
and let $W_{a,b}$ be the unique harmonic function in $P_{a,b}$ such that
\[
W_{a,b}=0 \text{ on }\partial\Gamma,
\qquad
W_{a,b}=st \text{ on }\partial H_{a,b},
\qquad
\int_{P_{a,b}} |\nabla W_{a,b}|^2<\infty.
\]
Define
\[
\mathcal F(a,b):=\int_{P_{a,b}} |\nabla W_{a,b}|^2\,ds\,dt.
\]

Define also the full one-hole capacity
\[
\mathcal C(a,b):=\mathcal F(a,b)+\int_{H_{a,b}}|\nabla(st)|^2\,ds\,dt
=\mathcal F(a,b)+\pi\left(a^2+b^2+\frac12\right).
\]
Thus $\mathcal C(a,1)=\mathcal C(a)$.

By construction,
\[
\mathcal F(a,1)=\mathcal E(a)
\qquad (a\ge 1),
\]
where $\mathcal E(a)$ is the one-hole axis-tangent energy from Section~\ref{sec:side-branch}. Moreover,
\[
\mathcal F(a,b)=\mathcal F(b,a)
\qquad (a,b\ge 1),
\]
by the symmetry $(s,t)\mapsto(t,s)$.

For later use, let $Y_{a,b}$ and $X_{a,b}$ denote the unique harmonic functions in $P_{a,b}$ such that
\[
Y_{a,b}=0 \text{ on }\partial\Gamma,
\qquad
Y_{a,b}=t \text{ on }\partial H_{a,b},
\qquad
\int_{P_{a,b}} |\nabla Y_{a,b}|^2<\infty,
\]
\[
X_{a,b}=0 \text{ on }\partial\Gamma,
\qquad
X_{a,b}=s \text{ on }\partial H_{a,b},
\qquad
\int_{P_{a,b}} |\nabla X_{a,b}|^2<\infty.
\]
Set
\[
J_y(a,b):=\int_{P_{a,b}} |\nabla Y_{a,b}|^2,
\qquad
J_x(a,b):=\int_{P_{a,b}} |\nabla X_{a,b}|^2,
\]
\[
I_y(a,b):=\int_{P_{a,b}} \nabla W_{a,b}\cdot \nabla Y_{a,b},
\qquad
I_x(a,b):=\int_{P_{a,b}} \nabla W_{a,b}\cdot \nabla X_{a,b}.
\]
Let the far-field coefficients be defined by
\[
W_{a,b}(r,\theta)=c(a,b)r^{-2}\sin(2\theta)+O(r^{-4}),
\]
\[
Y_{a,b}(r,\theta)=d_y(a,b)r^{-2}\sin(2\theta)+O(r^{-4}),
\qquad
X_{a,b}(r,\theta)=d_x(a,b)r^{-2}\sin(2\theta)+O(r^{-4})
\]
as $r\to+\infty$ with $0<\theta<\pi/2$.

\begin{proposition}[General one-hole identities]\label{prop:one-hole-general-identities-v8}
For every $a,b\ge 1$,
\begin{align}
\mathcal F(a,b)
&=
\frac{\pi}{2}c(a,b)-\pi\Bigl(a^2+b^2+\frac12\Bigr),
\label{eq:F-general-c-v8}
\\
\intertext{and therefore}
\mathcal C(a,b)
&=\frac{\pi}{2}c(a,b),
\label{eq:C-general-c}
\\
 d_y(a,b)
&=
2a+\frac{2}{\pi}I_y(a,b),
\qquad
 d_x(a,b)
=
2b+\frac{2}{\pi}I_x(a,b),
\label{eq:dy-dx-general-v8}
\\
 I_y(a,b)&=aJ_y(a,b),
\qquad
 I_x(a,b)=bJ_x(a,b).
\label{eq:I-equals-aJ-general-v8}
\end{align}
Consequently,
\[
 d_y(a,b)=2a\Bigl(1+\frac{J_y(a,b)}{\pi}\Bigr),
\qquad
 d_x(a,b)=2b\Bigl(1+\frac{J_x(a,b)}{\pi}\Bigr).
\]
\end{proposition}

\begin{proof}
The proof is the same Green-identity computation as in Proposition~\ref{prop:exact-identities-v4}, now with the disk centered at $(a,b)$.

Applying Green's identity to $W_{a,b}$ and $q(s,t):=st$ on $P_{a,b}\cap B_R$ gives \eqref{eq:F-general-c-v8}, because
\[
\int_{H_{a,b}} (s^2+t^2)\,ds\,dt
=
\pi\Bigl(a^2+b^2+\frac12\Bigr).
\]
Likewise, applying Green's second identity to $Y_{a,b}$ and $q(s,t)=st$ yields the first formula in \eqref{eq:dy-dx-general-v8}, because
\[
\int_{\partial H_{a,b}} t\,\partial_n(st)\,d\sigma
=
-\int_{H_{a,b}} \nabla\!\cdot\!\bigl(t\nabla(st)\bigr)\,ds\,dt
=
-\int_{H_{a,b}} s\,ds\,dt
=
-a\pi,
\]
where $n$ is the outer normal of the perforated domain. The formula for $d_x(a,b)$ is identical after interchanging $s$ and $t$.

For the relation $I_y(a,b)=aJ_y(a,b)$, consider
\[
\Psi_y:=W_{a,b}-aY_{a,b}.
\]
Then $\Psi_y$ is harmonic in $P_{a,b}$, vanishes on the axes, and satisfies
\[
\Psi_y=(s-a)t
\qquad\text{on }\partial H_{a,b}.
\]
Repeating the proof of \eqref{eq:F-general-c-v8} with the harmonic polynomial $(s-a)t$ in place of $st$ gives
\[
\int_{P_{a,b}} |\nabla\Psi_y|^2
=
\frac{\pi}{2}\bigl(c(a,b)-a d_y(a,b)\bigr)-\pi\Bigl(b^2+\frac12\Bigr).
\]
On the other hand,
\[
\int_{P_{a,b}} |\nabla\Psi_y|^2
=
\mathcal F(a,b)-2aI_y(a,b)+a^2J_y(a,b).
\]
Substituting \eqref{eq:F-general-c-v8} and the first formula in \eqref{eq:dy-dx-general-v8}, and simplifying, gives
\[
a^2J_y(a,b)=aI_y(a,b).
\]
Since $a\ge 1$, this proves $I_y(a,b)=aJ_y(a,b)$. The identity for $I_x(a,b)$ is obtained in the same way from $\Psi_x:=W_{a,b}-bX_{a,b}$.
\end{proof}

\begin{lemma}[Half-plane comparison for $J_y$]\label{lem:half-plane-comparison-Jy-v8}
For every $a,b\ge 1$,
\[
J_y(a,b)\ge J_{\mathrm{hp}}(b),
\]
where $J_{\mathrm{hp}}(b)$ is the minimum Dirichlet energy in
\[
\Omega_b^{\mathrm{hp}}:=\{(x,y):y>0\}\setminus \overline{B_1((0,b))}
\]
among functions equal to $0$ on $y=0$ and to $y$ on $\partial B_1((0,b))$.
\end{lemma}

\begin{proof}
Translate by $x=s-a$, $y=t$. Then $J_y(a,b)$ is the minimum Dirichlet energy in
\[
\Omega_{a,b}:=\{(x,y):x>-a,\ y>0\}\setminus \overline{B_1((0,b))}
\]
with boundary values $0$ on $x=-a$ and $y=0$, and $y$ on the circle. If $u$ is admissible on $\Omega_{a,b}$, extend it by $0$ to the larger half-plane domain $\Omega_b^{\mathrm{hp}}$. The extension is admissible there and has the same energy. Hence $J_y(a,b)\ge J_{\mathrm{hp}}(b)$.
\end{proof}

\begin{proposition}[A uniform lower bound for the auxiliary energies]\label{prop:Jy-greater-than-2pi-v8}
For every $a,b\ge 1$,
\[
J_y(a,b)>2\pi,
\qquad
J_x(a,b)>2\pi.
\]
\end{proposition}

\begin{proof}
By symmetry it suffices to prove the statement for $J_y$.
By Lemma~\ref{lem:half-plane-comparison-Jy-v8}, it is enough to prove
\[
J_{\mathrm{hp}}(b)>2\pi
\qquad (b\ge 1).
\]
For $b=1$, this is exactly the tangent half-plane energy from Proposition~\ref{prop:J-lower-bound-v4}:
\[
J_{\mathrm{hp}}(1)=\mathcal J_\infty=\pi\Bigl(\frac{\pi^2}{3}-1\Bigr)>2\pi.
\]
Now let $b>1$. Oddly reflect across $y=0$. This doubles the energy and gives the exterior problem in
\[
\mathbb C\setminus\bigl(\overline{B_1(ib)}\cup\overline{B_1(-ib)}\bigr)
\]
with boundary data $Y$ on both circles. Write
\[
b=\cosh\tau_0,
\qquad
\alpha:=\sinh\tau_0=\sqrt{b^2-1}.
\]
In bipolar coordinates
\[
X=\frac{\alpha\sin\sigma}{\cosh\tau-\cos\sigma},
\qquad
Y=\frac{\alpha\sinh\tau}{\cosh\tau-\cos\sigma},
\]
the exterior of the two circles is conformally mapped onto the strip
\[
S_{\tau_0}:=\{(\tau,\sigma): |\tau|<\tau_0,\ -\pi<\sigma\le \pi\},
\]
with the circles corresponding to $\tau=\pm\tau_0$. On $\tau=\pm\tau_0$, the boundary data become $\pm f(\sigma)$, where
\[
f(\sigma)=\frac{\alpha^2}{\cosh\tau_0-\cos\sigma}.
\]
Using the standard Fourier expansion
\[
\frac{\sinh\tau_0}{\cosh\tau_0-\cos\sigma}
=
1+2\sum_{n\ge 1} e^{-n\tau_0}\cos(n\sigma),
\]
we obtain
\[
f(\sigma)=\alpha\Bigl(1+2\sum_{n\ge 1} e^{-n\tau_0}\cos(n\sigma)\Bigr).
\]
Hence the harmonic solution in the strip is
\[
U(\tau,\sigma)
=
\alpha\frac{\tau}{\tau_0}
+
2\alpha\sum_{n\ge 1} e^{-n\tau_0}
\frac{\sinh(n\tau)}{\sinh(n\tau_0)}\cos(n\sigma).
\]
By conformal invariance of the two-dimensional Dirichlet integral (not of the eigenvalue itself) and orthogonality of the Fourier modes,
\[
2J_{\mathrm{hp}}(b)
=
\frac{4\pi\alpha^2}{\tau_0}
+
8\pi\alpha^2\sum_{n\ge 1} n e^{-2n\tau_0}\coth(n\tau_0).
\]
Using $\coth(n\tau_0)\ge 1/(n\tau_0)$, we get
\[
2J_{\mathrm{hp}}(b)
\ge
\frac{4\pi\alpha^2}{\tau_0}
+
\frac{8\pi\alpha^2}{\tau_0}\sum_{n\ge 1} e^{-2n\tau_0}
=
\frac{4\pi\alpha^2}{\tau_0}\coth\tau_0
=
\frac{4\pi\sinh\tau_0\cosh\tau_0}{\tau_0}
=
\frac{2\pi\sinh(2\tau_0)}{\tau_0}.
\]
Since $\sinh(2\tau_0)>2\tau_0$ for every $\tau_0>0$, it follows that
\[
2J_{\mathrm{hp}}(b)>4\pi,
\]
so $J_{\mathrm{hp}}(b)>2\pi$. This proves $J_y(a,b)>2\pi$. The same argument after interchanging $s$ and $t$ gives $J_x(a,b)>2\pi$.
\end{proof}

\begin{corollary}[Quadratic lower bounds for the one-hole family]\label{cor:F-quadratic-lower-bound-v9}
For every $a,b\ge 1$,
\[
\mathcal F(a,b)\ge a^2 J_y(a,b),
\qquad
\mathcal F(a,b)\ge b^2 J_x(a,b).
\]
Consequently,
\[
\mathcal F(a,b)>2\pi a^2,
\qquad
\mathcal F(a,b)>2\pi b^2,
\]
and hence
\[
\mathcal F(a,b)>2\pi\max\{a^2,b^2\}.
\]
In particular,
\[
\mathcal E(1+\sqrt2)=\mathcal F(1+\sqrt2,1)>2\pi(3+2\sqrt2).
\]
\end{corollary}

\begin{proof}
By Cauchy--Schwarz and Proposition~\ref{prop:one-hole-general-identities-v8},
\[
I_y(a,b)^2
\le
\mathcal F(a,b)J_y(a,b)
\qquad\text{and}\qquad
I_y(a,b)=aJ_y(a,b).
\]
Since $J_y(a,b)>0$, this gives
\[
\mathcal F(a,b)\ge a^2 J_y(a,b).
\]
The same argument with $X_{a,b}$ yields
\[
\mathcal F(a,b)\ge b^2 J_x(a,b).
\]
Now apply Proposition~\ref{prop:Jy-greater-than-2pi-v8} to obtain the strict inequalities
\[
\mathcal F(a,b)>2\pi a^2,
\qquad
\mathcal F(a,b)>2\pi b^2.
\]
The final displayed bounds follow immediately.
\end{proof}

\begin{theorem}[Coordinatewise monotonicity of the one-hole family]\label{thm:F-coordinatewise-monotone-v8}
The function $\mathcal F(a,b)$ is strictly increasing in each coordinate on $[1,\infty)^2$.
Equivalently, if $A>a\ge 1$ and $B>b\ge 1$, then
\[
\mathcal F(A,b)>\mathcal F(a,b),
\qquad
\mathcal F(a,B)>\mathcal F(a,b).
\]
In particular,
\[
\mathcal E(A)=\mathcal F(A,1)>\mathcal F(a,1)=\mathcal E(a)
\qquad (A>a\ge 1).
\]
\end{theorem}

\begin{proof}
Fix $a,b\ge 1$ and $\delta>0$. Define
\[
\Phi(s,t):=W_{a+\delta,b}(s+\delta,t)
\qquad\text{on }P_{a,b}.
\]
This is well-defined because
\[
(s-a)^2+(t-b)^2\ge 1
\Longrightarrow
(s+\delta-(a+\delta))^2+(t-b)^2\ge 1.
\]
Thus $\Phi$ is harmonic in $P_{a,b}$. On $\partial H_{a,b}$,
\[
\Phi(s,t)=W_{a+\delta,b}(s+\delta,t)=(s+\delta)t=st+\delta t=W_{a,b}(s,t)+\delta Y_{a,b}(s,t).
\]
On $t=0$, all three functions vanish. On $s=0$, one has
\[
\Phi(0,t)=W_{a+\delta,b}(\delta,t)\ge 0=W_{a,b}(0,t)+\delta Y_{a,b}(0,t).
\]
Hence $\Phi-(W_{a,b}+\delta Y_{a,b})$ is harmonic in $P_{a,b}$ and nonnegative on the whole boundary. The maximum principle on truncated domains and passage to the limit give
\[
W_{a+\delta,b}(s+\delta,t)\ge W_{a,b}(s,t)+\delta Y_{a,b}(s,t)
\qquad\text{in }P_{a,b}.
\]
Comparing the far-field coefficients along any fixed interior ray yields
\[
c(a+\delta,b)\ge c(a,b)+\delta d_y(a,b).
\]
Using Proposition~\ref{prop:one-hole-general-identities-v8}, we obtain
\[
\mathcal F(a+\delta,b)-\mathcal F(a,b)
\ge
\frac{\pi}{2}\delta d_y(a,b)-\pi\bigl((a+\delta)^2-a^2\bigr)
=
a\bigl(J_y(a,b)-\pi\bigr)\delta-\pi\delta^2.
\]
By Proposition~\ref{prop:Jy-greater-than-2pi-v8}, this gives
\[
\mathcal F(a+\delta,b)-\mathcal F(a,b)
>
\pi a\,\delta-\pi\delta^2.
\]
Now let $A>a$. Choose a partition
\[
a=x_0<x_1<\cdots<x_N=A
\]
so fine that $x_i-x_{i-1}<x_{i-1}$ for every $i$. Then each increment is strictly positive, and summing over the partition gives
\[
\mathcal F(A,b)>\mathcal F(a,b).
\]
The monotonicity in $b$ is proved in exactly the same way after interchanging $s$ and $t$.
The final statement follows from $\mathcal E(a)=\mathcal F(a,1)$.
\end{proof}

\begin{corollary}[Coordinatewise monotonicity of the full one-hole capacity]\label{cor:C-coordinatewise-monotone}
The function $\mathcal C(a,b)$ is strictly increasing in each coordinate on $[1,\infty)^2$.
\end{corollary}
\begin{proof}
Fix $\delta>0$. The shifted function $W_{a+\delta,b}(s+\delta,t)$ is harmonic on $P_{a,b}$. On the circular boundary it equals $W_{a,b}+\delta Y_{a,b}$; on $t=0$ all terms vanish, while on $s=0$ the shifted function is nonnegative. The truncated-domain maximum-principle argument used in the proof of Theorem~\ref{thm:F-coordinatewise-monotone-v8} therefore gives
\[
W_{a+\delta,b}(s+\delta,t)\ge W_{a,b}(s,t)+\delta Y_{a,b}(s,t).
\]
Comparison of the first far-field coefficients yields
\[
c(a+\delta,b)\ge c(a,b)+\delta d_y(a,b)>c(a,b),
\]
because Proposition~\ref{prop:one-hole-general-identities-v8} gives $d_y(a,b)>0$. Equation~\eqref{eq:C-general-c} proves strict increase in $a$. Interchanging $s$ and $t$ proves strict increase in $b$.
\end{proof}

\begin{proposition}[Reflected comparison with the one-hole family]\label{prop:same-corner-reflected-comparison-v9}

For $a,b\ge1$, define
\[
K_{a,b}:=\bigcup_{\sigma_1,\sigma_2\in\{\pm1\}}\overline{B_1((\sigma_1a,\sigma_2b))}.
\]
For an admissible same-corner pair $c_i=(a_i,b_i)$, set
\[
K(c_1,c_2):=K_{a_1,b_1}\cup K_{a_2,b_2}.
\]
Then
\[
\mathfrak C(K_{a,b},XY)=4\mathcal C(a,b),
\qquad
\mathfrak C(K(c_1,c_2),XY)=4\mathcal C_2(c_1,c_2).
\]
Moreover,
\[
\mathcal C_2(c_1,c_2)\ge
\max\{\mathcal C(a_1,b_1),\mathcal C(a_2,b_2)\}.
\]

\end{proposition}

\begin{proof}

The identities follow by odd reflection of the full capacity problem across both axes. In each reflected quadrant the capacitary minimizer equals $st$ on the disk itself and the corresponding harmonic corrector outside it; the four quadrants therefore contribute four copies of the full quadrant capacity. Conversely, uniqueness of the capacity minimizer and odd symmetry recover the quadrant problem.

For $i=1,2$ one has $K_{a_i,b_i}\subset K(c_1,c_2)$. The full capacity is monotone with respect to the constrained compact set: every competitor satisfying $V=XY$ quasi-everywhere on $K(c_1,c_2)$ also satisfies the constraint on $K_{a_i,b_i}$. Hence
\[
\mathfrak C(K(c_1,c_2),XY)\ge \mathfrak C(K_{a_i,b_i},XY),
\]
and division by four gives the stated bound. Notice that no zero extension of an exterior corrector is used here.

\end{proof}

\begin{theorem}[Scalar reduction of the same-corner exclusion]\label{thm:same-corner-scalar-reduction-v9}

For every admissible same-corner pair $(c_1,c_2)$,
\[
\mathcal C_2(c_1,c_2)\ge \mathcal C(1+\sqrt2).
\]

\end{theorem}

\begin{proof}
Let $(c_1,c_2)$ be admissible, with $c_i=(a_i,b_i)$. If all four coordinates satisfied
\[
a_i<1+\sqrt2,
\qquad
b_i<1+\sqrt2
\qquad (i=1,2),
\]
then both centers would lie in the square $[1,1+\sqrt2)^2$, whose diameter is strictly less than $2$. This would contradict the admissibility condition $|c_1-c_2|\ge 2$. Hence for at least one index $i\in\{1,2\}$ one has either $a_i\ge 1+\sqrt2$ or $b_i\ge 1+\sqrt2$.

By Proposition~\ref{prop:same-corner-reflected-comparison-v9},
\[
\mathcal C_2(c_1,c_2)\ge \mathcal C(a_i,b_i).
\]
If $a_i\ge1+\sqrt2$, Corollary~\ref{cor:C-coordinatewise-monotone} gives
\[
\mathcal C(a_i,b_i)\ge\mathcal C(1+\sqrt2,1)=\mathcal C(1+\sqrt2).
\]
If $b_i\ge1+\sqrt2$, symmetry and the same monotonicity give the identical conclusion. This proves the claim.

\end{proof}

\begin{proposition}[An explicit upper bound for the corner-tangent one-hole cell]\label{prop:E1-explicit-upper-v10}
One has
\[
\mathcal E(1)\le \frac{55}{18}+\frac{275\pi}{64}.
\]
In particular,
\[
\mathcal E(1)<\pi(3+2\sqrt2).
\]
\end{proposition}

\begin{proof}
Let
\[
P_{1,1}:=\Gamma\setminus \overline{B_1((1,1))}.
\]
Consider the explicit admissible function
\[
V(s,t):=\frac{st}{\bigl((s-1)^2+(t-1)^2\bigr)^{5/4}}
\qquad\text{on }P_{1,1}.
\]
On $\partial\Gamma$ one has $V=0$, and on $\partial B_1((1,1))$ one has $V=st$, since the denominator equals $1$ there. Moreover,
\[
V(s,t)=O(\rho^{-1/2}),
\qquad
|\nabla V(s,t)|=O(\rho^{-3/2})
\qquad (\rho\to+\infty),
\]
where
\[
\rho:=\sqrt{(s-1)^2+(t-1)^2},
\]
so $V$ has finite Dirichlet energy on $P_{1,1}$. Therefore $V$ is admissible for the variational problem defining $\mathcal E(1)$, and hence
\[
\mathcal E(1)\le \int_{P_{1,1}} |\nabla V|^2\,ds\,dt.
\]

Introduce polar coordinates around $(1,1)$:
\[
s=1+\rho\cos\theta,
\qquad
t=1+\rho\sin\theta,
\qquad
\rho\ge 1.
\]
For fixed $\rho\ge 1$, the admissible angular interval is
\[
I_\rho:=[-\phi(\rho),\,\tfrac{\pi}{2}+\phi(\rho)],
\qquad
\phi(\rho):=\arcsin(1/\rho).
\]
In these coordinates,
\[
V(\rho,\theta)=\rho^{-5/2}+(\cos\theta+\sin\theta)\rho^{-3/2}
+\sin\theta\cos\theta\,\rho^{-1/2},
\]
so
\[
V_\rho=
-\frac{5}{2}\rho^{-7/2}
-\frac{3}{2}(\cos\theta+\sin\theta)\rho^{-5/2}
-\frac{1}{2}\sin\theta\cos\theta\,\rho^{-3/2},
\]
\[
V_\theta=
(\cos\theta-\sin\theta)\rho^{-3/2}
+(\cos^2\theta-\sin^2\theta)\rho^{-1/2}.
\]
Hence
\[
\int_{P_{1,1}} |\nabla V|^2\,ds\,dt
=
\int_1^\infty \int_{I_\rho}
\Bigl(V_\rho^2+\rho^{-2}V_\theta^2\Bigr)\rho\,d\theta\,d\rho.
\]
A direct trigonometric expansion followed by the $\theta$-integration gives
\[
\begin{aligned}
\int_{I_\rho}
\Bigl(V_\rho^2+\rho^{-2}V_\theta^2\Bigr)\rho\,d\theta
&=
\frac{1}{192\rho^6}
\Bigl(
204\rho^4\phi(\rho)+51\pi\rho^4
+628\rho^2\sqrt{\rho^2-1}\\
&\qquad{}+1248\rho^2\phi(\rho)+312\pi\rho^2+1248\rho^2
+2840\sqrt{\rho^2-1}\\
&\qquad{}+2400\phi(\rho)+1600+600\pi
\Bigr).
\end{aligned}
\]
Therefore
\[
\int_{P_{1,1}} |\nabla V|^2\,ds\,dt
=
\frac{1}{192}\Bigl(
\begin{aligned}
&204A_2+51\pi B_2+628C_4+1248A_4+312\pi B_4+1248B_4\\
&\qquad{}+2840C_6+2400A_6+1600B_6+600\pi B_6
\end{aligned}
\Bigr),
\]
where
\[
A_k:=\int_1^\infty \rho^{-k}\arcsin(1/\rho)\,d\rho,
\qquad
C_k:=\int_1^\infty \rho^{-k}\sqrt{\rho^2-1}\,d\rho,
\]
\[
B_k:=\int_1^\infty \rho^{-k}\,d\rho.
\]
Using $u=1/\rho$, integrations by parts, and the Beta-function evaluation of the square-root integrals, one finds
\[
A_2=\int_0^1 \arcsin u\,du=\frac{\pi}{2}-1,
\qquad
A_4=\int_0^1 u^2\arcsin u\,du=\frac{\pi}{6}-\frac{2}{9},
\]
\[
A_6=\int_0^1 u^4\arcsin u\,du=\frac{\pi}{10}-\frac{8}{75},
\qquad
C_4=\int_0^1 u\sqrt{1-u^2}\,du=\frac{1}{3},
\]
\[
C_6=\int_0^1 u^3\sqrt{1-u^2}\,du=\frac{2}{15}.
\]
and
\[
B_2=1,
\qquad
B_4=\frac{1}{3},
\qquad
B_6=\frac{1}{5}.
\]
Substituting these values and simplifying yields
\[
\int_{P_{1,1}} |\nabla V|^2\,ds\,dt
=
\frac{55}{18}+\frac{275\pi}{64}.
\]
Thus
\[
\mathcal E(1)\le \frac{55}{18}+\frac{275\pi}{64}.
\]

Finally,
\[
\frac{55}{18}+\frac{275\pi}{64}\approx 16.555
\qquad\text{and}\qquad
\pi(3+2\sqrt2)\approx 18.311,
\]
so indeed
\[
\mathcal E(1)<\pi(3+2\sqrt2).
\]
\end{proof}

\begin{corollary}[The scalar threshold inequality]\label{cor:scalar-threshold-v10}
One has
\[
\mathcal E(1+\sqrt2)>2\mathcal E(1).
\]
\end{corollary}

\begin{proof}
By Corollary~\ref{cor:F-quadratic-lower-bound-v9},
\[
\mathcal E(1+\sqrt2)=\mathcal F(1+\sqrt2,1)>2\pi(1+\sqrt2)^2=2\pi(3+2\sqrt2).
\]
On the other hand, Proposition~\ref{prop:E1-explicit-upper-v10} gives
\[
\mathcal E(1)<\pi(3+2\sqrt2).
\]
Combining the two inequalities yields
\[
\mathcal E(1+\sqrt2)>2\mathcal E(1).
\]
\end{proof}

\begin{corollary}[Same-corner cell exclusion]\label{cor:same-corner-cell-exclusion-v10}

For every admissible same-corner pair $(c_1,c_2)$,
\[
\mathcal C_2(c_1,c_2)>2\mathcal C(1).
\]

\end{corollary}

\begin{proof}

Set $A:=1+\sqrt2$. By Theorem~\ref{thm:same-corner-scalar-reduction-v9},
$\mathcal C_2(c_1,c_2)\ge\mathcal C(A)$. Corollary~\ref{cor:scalar-threshold-v10} gives
$\mathcal E(A)>2\mathcal E(1)$. Since
\[
\mathcal C(A)-2\mathcal C(1)
=\bigl(\mathcal E(A)-2\mathcal E(1)\bigr)
+\pi\left(A^2-\frac72\right)>0,
\]
because $A^2=3+2\sqrt2>7/2$, the desired strict inequality follows.

\end{proof}

\subsection{Same-corner small-hole asymptotic reduction}

Having obtained the strict cell inequality, we return to the physical square and transfer it to the first eigenvalue by the capacity expansion, uniformly on compact parameter sets.

For an admissible pair $(c_1,c_2)$ with $c_i=(a_i,b_i)$, define the physical same-corner configuration
\[
\Omega_r(c_1,c_2):=
Q\setminus\Bigl(
\overline{B_r(1-a_1r,\,1-b_1r)}
\cup
\overline{B_r(1-a_2r,\,1-b_2r)}
\Bigr).
\]
Equivalently, both holes lie in the top-right $O(r)$ corner layer, with scaled centers $c_1,c_2$.

\begin{proposition}[Same-corner small-hole asymptotic reduction]\label{prop:same-corner-small-hole-asymptotic-v11}
Let
\[
\mathcal A:=\{(c_1,c_2): c_i=(a_i,b_i),\ a_i,b_i\ge 1,\ |c_1-c_2|\ge 2\}.
\]
For every compact set $\mathcal K\subset \mathcal A$,

\[
\lambda_1\bigl(\Omega_r(c_1,c_2)\bigr)
=
\frac{\pi^2}{2}
+
\frac{\pi^4}{16}\,\mathcal C_2(c_1,c_2)\,r^4
+
o_{\mathcal K}(r^4)
\qquad (r\to 0),
\]

uniformly for $(c_1,c_2)\in \mathcal K$.
\end{proposition}

\begin{proof}
Reflect oddly across the top side $y=1$ and then oddly across the right side $x=1$, exactly as in the one-hole corner analysis. The outer domain becomes
\[
\widehat Q:=(-1,3)^2,
\]
and the two holes become the interior compact defect
\[
\widehat K_r(c_1,c_2)=(1,1)+rK(c_1,c_2).
\]
Thus $\lambda_1(\Omega_r(c_1,c_2))$ is identified with the perturbation, under removal of the compact set $\widehat K_r(c_1,c_2)$, of the simple Dirichlet eigenvalue
\[
\lambda_0=\frac{\pi^2}{2}
\]
of $-\Delta$ on $\widehat Q$. The $L^2(\widehat Q)$-normalized eigenfunction on this reflected square is
\[
\widehat u_0(x,y)=\frac12\cos\frac{\pi x}{2}\cos\frac{\pi y}{2}.
\]

Near $(1,1)$ one has
\[
\widehat u_0(1+rX,1+rY)=\frac{\pi^2}{8}r^2XY+O(r^4)
\]
uniformly for $(X,Y)$ in bounded sets. Hence the first nonzero homogeneous term is $P(X,Y)=(\pi^2/8)XY$. Lemma~\ref{lem:full-capacity-blowup} and the simple-eigenvalue capacity expansion therefore give
\[
\lambda_1\bigl(\Omega_r(c_1,c_2)\bigr)
=
\frac{\pi^2}{2}
+
\frac{\pi^4}{64}\,\mathfrak C(K(c_1,c_2),XY)\,r^4
+
o_{\mathcal K}(r^4).
\]
By Proposition~\ref{prop:same-corner-reflected-comparison-v9},
$\mathfrak C(K(c_1,c_2),XY)=4\mathcal C_2(c_1,c_2)$. This gives the coefficient $(\pi^4/16)\mathcal C_2(c_1,c_2)$ and proves the stated expansion uniformly on compact parameter sets.

\end{proof}

\begin{theorem}[Same-corner branch exclusion]\label{thm:same-corner-branch-complete-v11}
For every compact set $\mathcal K\subset\mathcal A$,

\[
\lambda_1\bigl(\Omega_r(c_1,c_2)\bigr)
=
\frac{\pi^2}{2}
+
\frac{\pi^4}{16}\,\mathcal C_2(c_1,c_2)\,r^4
+
o_{\mathcal K}(r^4)
>
\frac{\pi^2}{2}
+
\frac{\pi^4}{8}\,\mathcal C(1)\,r^4
+
o_{\mathcal K}(r^4),
\]

uniformly for $(c_1,c_2)\in\mathcal K$. In particular, no same-corner configuration can be asymptotically minimizing.
\end{theorem}

\begin{proof}
The asymptotic expansion is Proposition~\ref{prop:same-corner-small-hole-asymptotic-v11}. By Corollary~\ref{cor:same-corner-cell-exclusion-v10},
\[
\mathcal C_2(c_1,c_2)>2\mathcal C(1)
\]
for every admissible same-corner pair. Multiplication by $\pi^4r^4/16$ gives the strict coefficient inequality. The conclusion follows.
\end{proof}

\section{Finite element validation}\label{sec:numerics}

We next provide a numerical validation of the branch ordering predicted by Theorem~\ref{thm:main-small-hole-minimizer-v12}. All code, CSV files, and figures used in this section are archived in the benchmark dataset~\cite{Zhang2026DatasetFEM}. The numerics are not part of the proof; their role is to provide a transparent, reproducible check of the four representative two-hole branches singled out by the analysis. The new global localization argument is analytic and does not require a numerical scan of intermediate placements, so the benchmark remains focused on these representative branches.

The archived benchmark is a conforming boundary-fitted P1 finite element computation. The outer square is represented polygonally with eight boundary segments per side, and each circular obstacle is approximated by a $32$-gon (equivalently, the dataset parameter \texttt{quad\_segs} equals $8$). A constrained Delaunay triangulation of the resulting polygonal domain is generated, followed by two uniform red refinements, giving the three displayed mesh levels (the initial mesh and two refinements). Stiffness and mass matrices are assembled on the P1 space, homogeneous Dirichlet conditions are imposed on the outer boundary and on the obstacle boundaries by eliminating boundary degrees of freedom, and the smallest generalized eigenpair of the reduced stiffness--mass pencil is computed.

To keep tangent geometries robustly meshable and exactly reproducible, the benchmark introduces a very small inward geometric inset
\[
\varepsilon_{\mathrm{geom}}=5\times 10^{-4},
\]
so that nominally tangent configurations are meshed with an offset that is tiny compared with the tested radii. Accordingly, the FEM results should be read as a validation of branch ordering and relative scale rather than as an independent source of sharp asymptotic coefficients. This is particularly relevant for the adjacent-versus-opposite comparison, where the true gap is far smaller than the separation from the clearly non-minimizing branches.

The four representative two-hole branches used in the benchmark are:
\begin{enumerate}[label=\rm(\alph*)]
\item two holes tangent to adjacent corners,
\item two holes tangent to opposite corners,
\item two holes tangent to opposite sides at the center,
\item a clustered same-corner configuration with the two disks nearly tangent.
\end{enumerate}
Figure~\ref{fig:fem-geometries-v13} shows these four benchmark geometries.

\begin{figure}[t]
\centering
\includegraphics[width=0.74\textwidth]{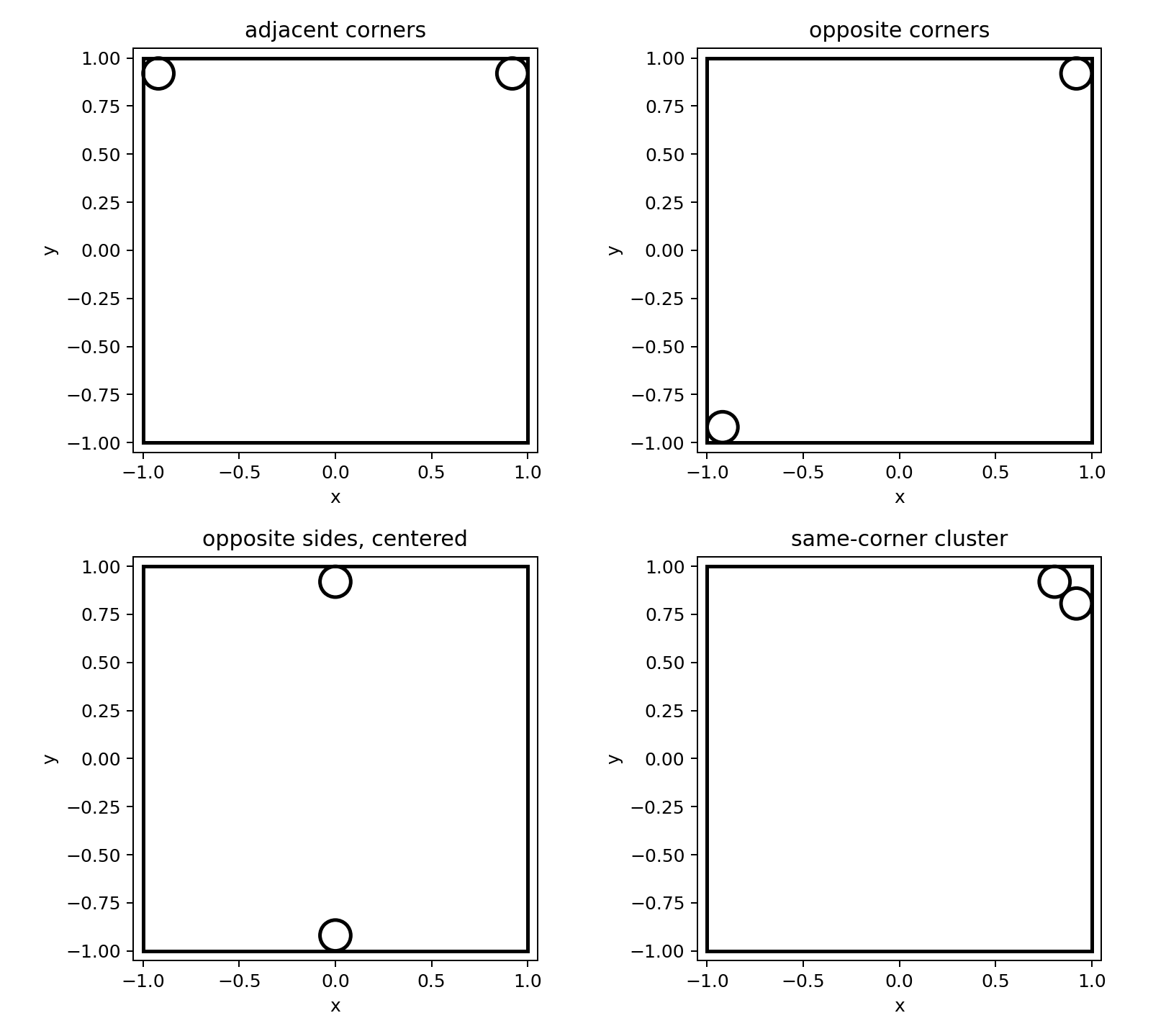}
\caption{Representative two-hole geometries used in the FEM validation package: adjacent corners, opposite corners, opposite sides at the center, and a clustered same-corner configuration with the two disks nearly tangent.}
\label{fig:fem-geometries-v13}
\end{figure}

For the representative radius $r=0.08$, the dataset also contains the corresponding first FEM eigenfunctions. Figure~\ref{fig:fem-eigenfunctions-v14} shows the first mode on the four benchmark geometries. The plots agree with the ordering in Table~\ref{tab:fem-case-table-v13}: the opposite-sides configuration constricts the domain much more severely than the corner branches, while the same-corner cluster is already less favorable than the adjacent- and opposite-corner competitors.

\begin{figure}[t]
\centering
\begin{minipage}{0.48\textwidth}
\centering
\includegraphics[width=0.92\textwidth]{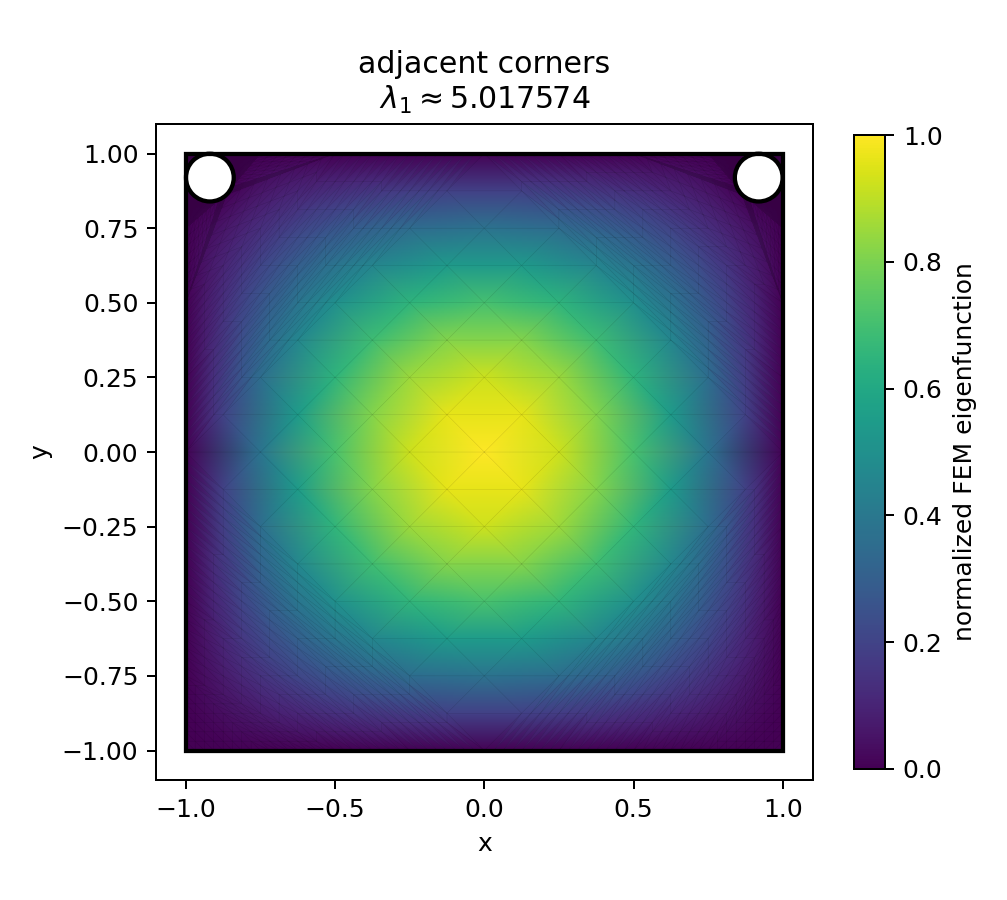}\\[-1mm]
{\footnotesize adjacent corners}
\end{minipage}\hfill
\begin{minipage}{0.48\textwidth}
\centering
\includegraphics[width=0.92\textwidth]{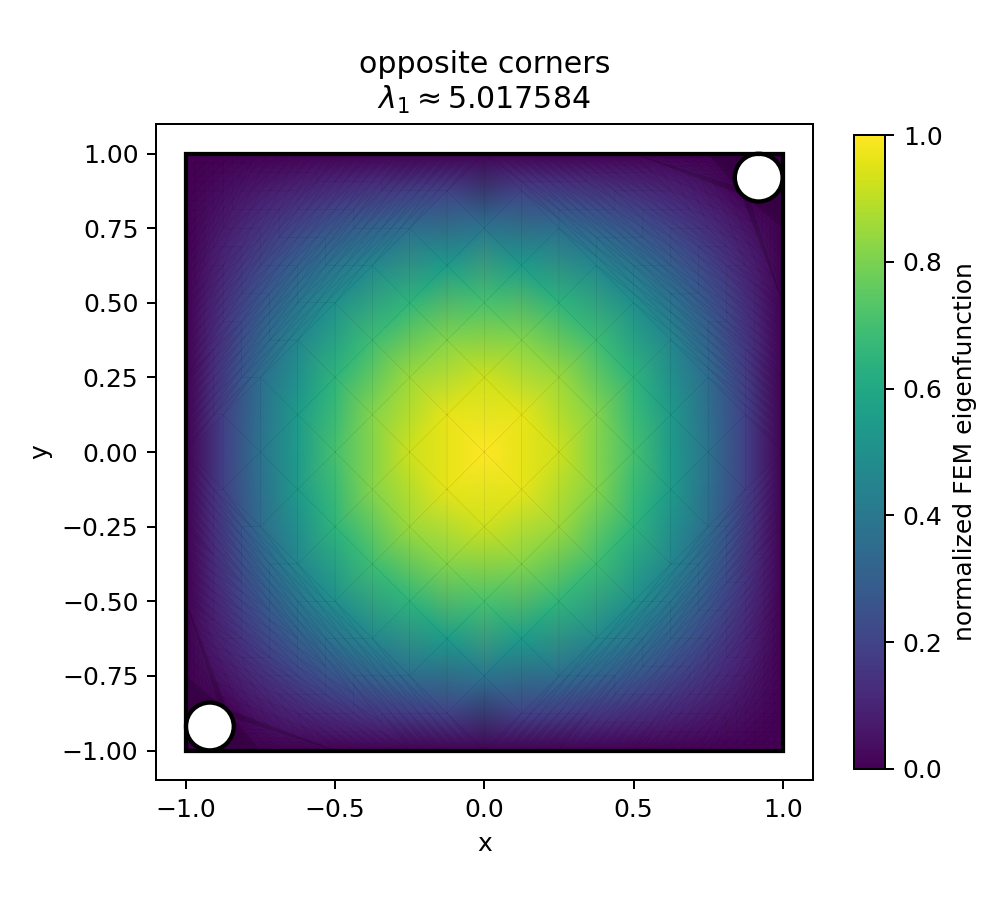}\\[-1mm]
{\footnotesize opposite corners}
\end{minipage}

\medskip
\begin{minipage}{0.48\textwidth}
\centering
\includegraphics[width=0.92\textwidth]{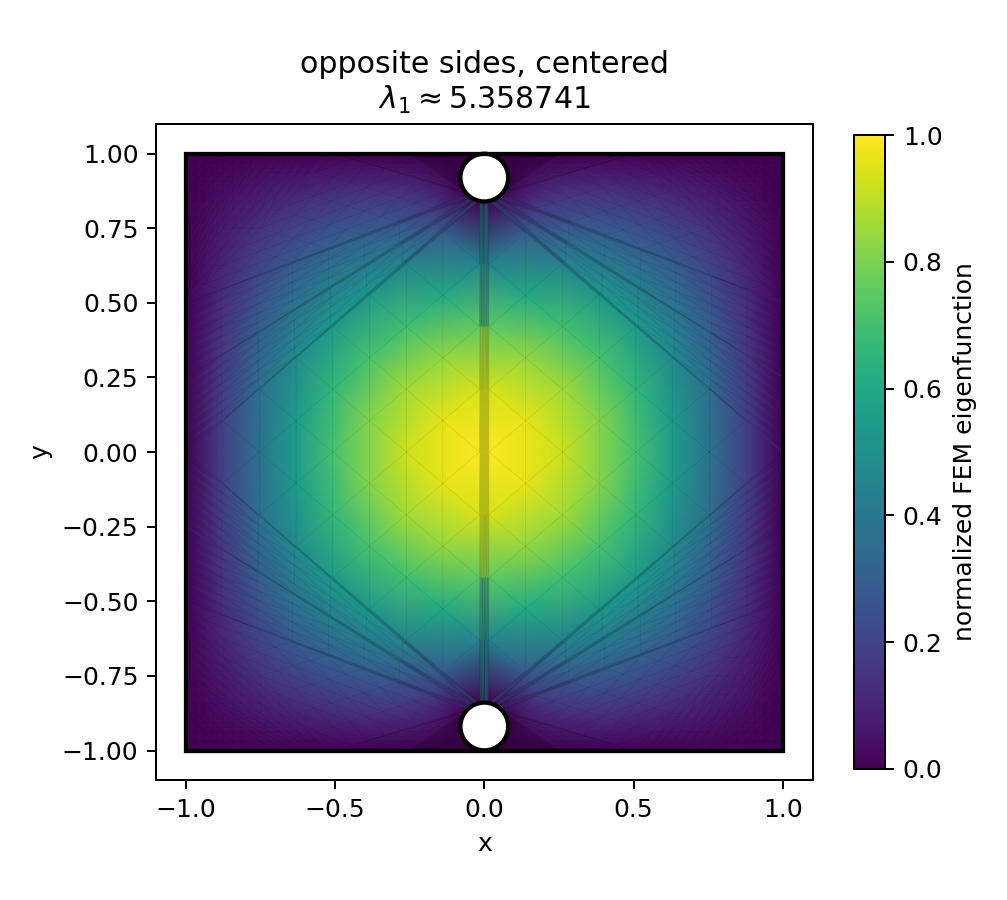}\\[-1mm]
{\footnotesize opposite sides at the center}
\end{minipage}\hfill
\begin{minipage}{0.48\textwidth}
\centering
\includegraphics[width=0.92\textwidth]{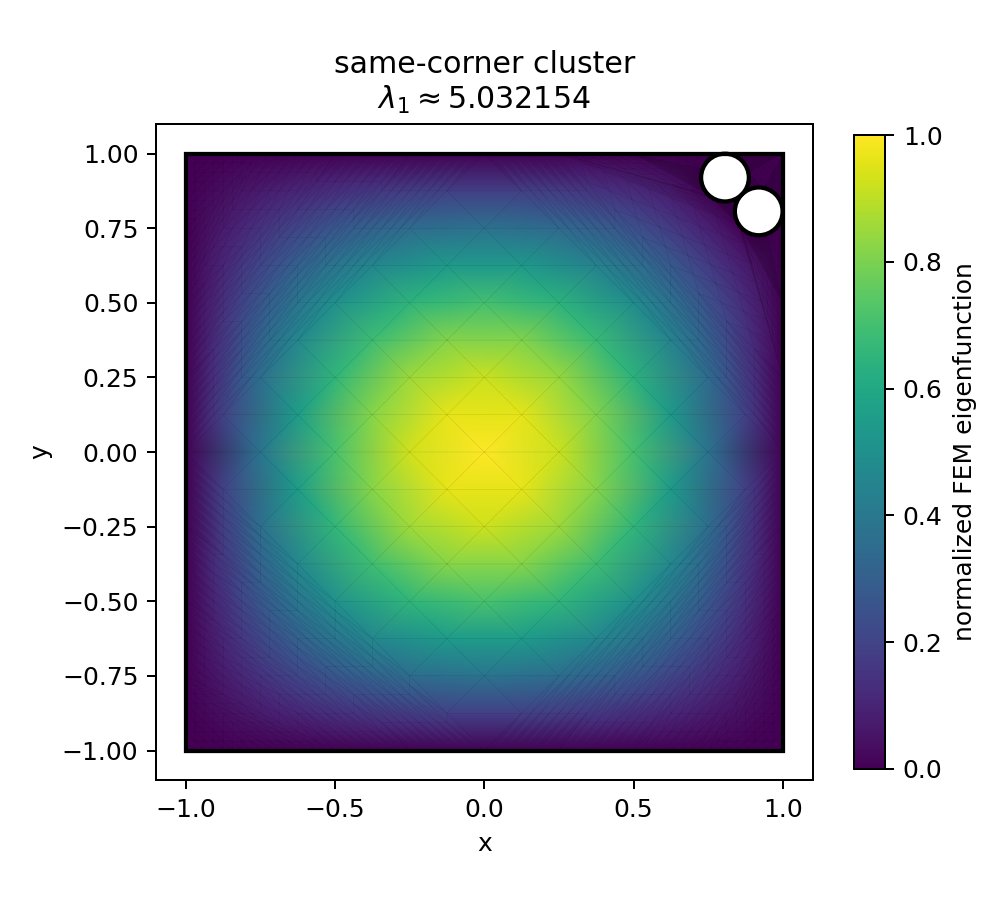}\\[-1mm]
{\footnotesize same-corner cluster}
\end{minipage}
\caption{First FEM eigenfunctions for the four representative benchmark geometries at $r=0.08$.}
\label{fig:fem-eigenfunctions-v14}
\end{figure}

As a basic consistency check, the FEM package also records the empty-square reference against the exact value $\lambda_1(Q)=\pi^2/2\approx 4.934802200545$. The corresponding convergence table appears in Table~\ref{tab:fem-convergence-v13}. The approximation decreases monotonically toward the exact value, as expected for the conforming Dirichlet FEM.

\begin{table}[htbp]
\centering
\caption{Empty-square FEM convergence from the archived FEM benchmark package.}
\label{tab:fem-convergence-v13}
\begin{tabular}{cccc}
\toprule
refinement level & nodes & triangles & $\lambda_1$ \\
\midrule
1 & 93 & 120 & 6.231004 \\
2 & 305 & 480 & 5.225946 \\
3 & 1089 & 1920 & 5.006070 \\
\bottomrule
\end{tabular}
\end{table}

For the four representative two-hole branches, Table~\ref{tab:fem-case-table-v13} reports the values produced by the benchmark package for $r\in\{0.07,0.08,0.09\}$. The same ordering is observed at all three radii:
\[
\lambda_{\mathrm{adjacent}}
<
\lambda_{\mathrm{opposite}}
<
\lambda_{\mathrm{cluster}}
<
\lambda_{\mathrm{opp-side}}.
\]
In particular, the adjacent-corner branch is always better than the opposite-corner branch, while both corner branches are substantially better than the same-corner clustered and opposite-sides competitors. The numerical gap between adjacent and opposite corners is tiny, but the separation from the non-minimizing branches is much larger and remains clearly visible on the absolute scale.

\begin{table}[htbp]
\centering
\caption{Representative FEM eigenvalues for the four benchmark branches.}
\label{tab:fem-case-table-v13}
\begin{tabular}{ccccc}
\toprule
$r$ & adjacent corners & opposite corners & same-corner cluster & opposite sides \\
\midrule
0.07 & 5.012947 & 5.012949 & 5.021807 & 5.281190 \\
0.08 & 5.017574 & 5.017584 & 5.032154 & 5.358741 \\
0.09 & 5.024208 & 5.024236 & 5.046506 & 5.446469 \\
\bottomrule
\end{tabular}
\end{table}

The adjacent-versus-opposite gap is small on the absolute scale, so it is more informative to display both the absolute branch values and the isolated gap. Figure~\ref{fig:fem-branch-ordering-v13} presents these two views side by side. The benchmark gives
\[
\lambda_{\mathrm{opp}}-\lambda_{\mathrm{adj}}
\approx
2.72\times 10^{-6},\ 1.01\times 10^{-5},\ 2.76\times 10^{-5}
\]
for $r=0.07,0.08,0.09$, respectively, again with the correct positive sign.

\begin{figure}[htbp]
\centering
\begin{minipage}{0.48\textwidth}
\centering
\includegraphics[width=\textwidth]{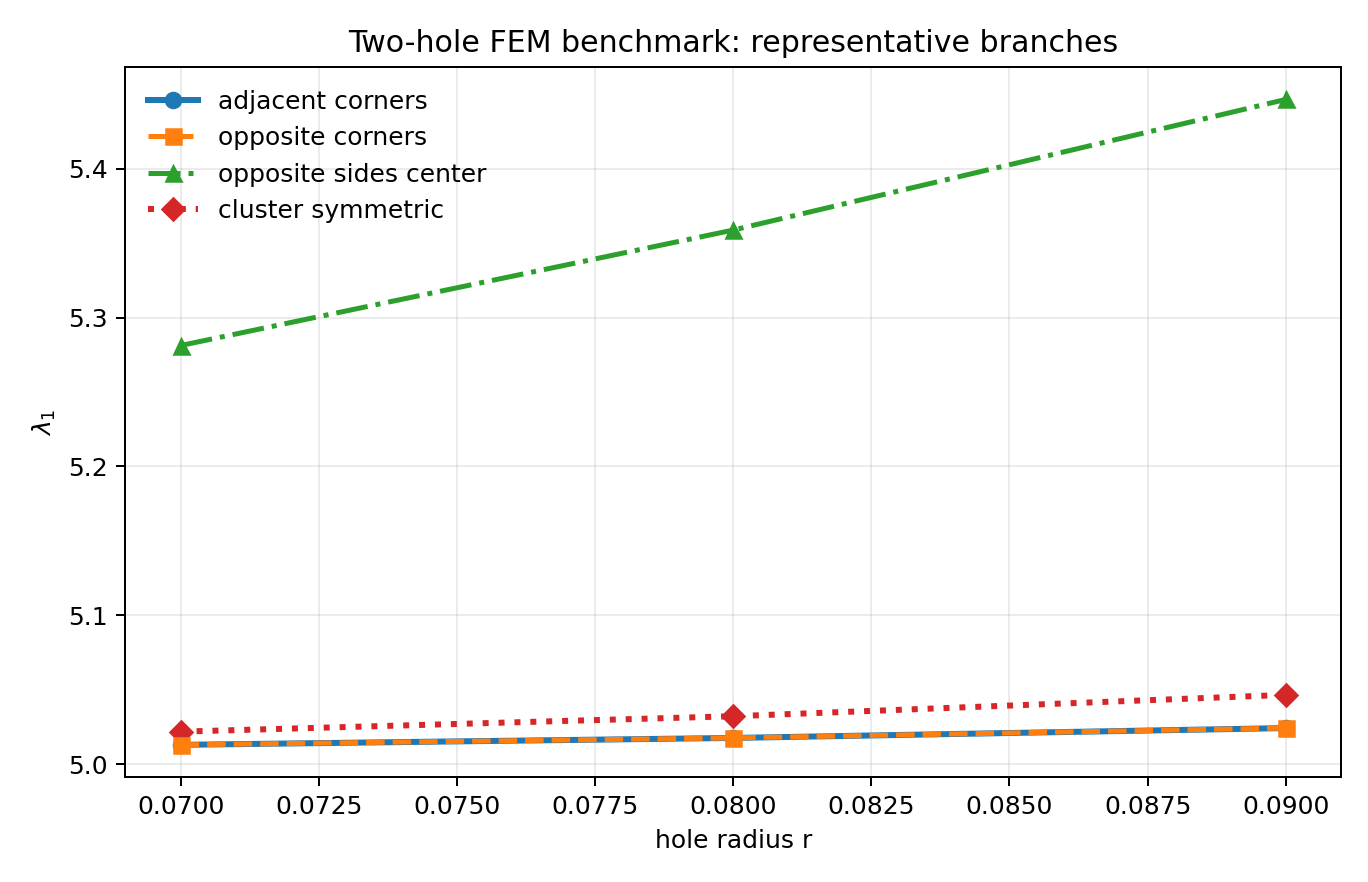}
\end{minipage}\hfill
\begin{minipage}{0.48\textwidth}
\centering
\includegraphics[width=\textwidth]{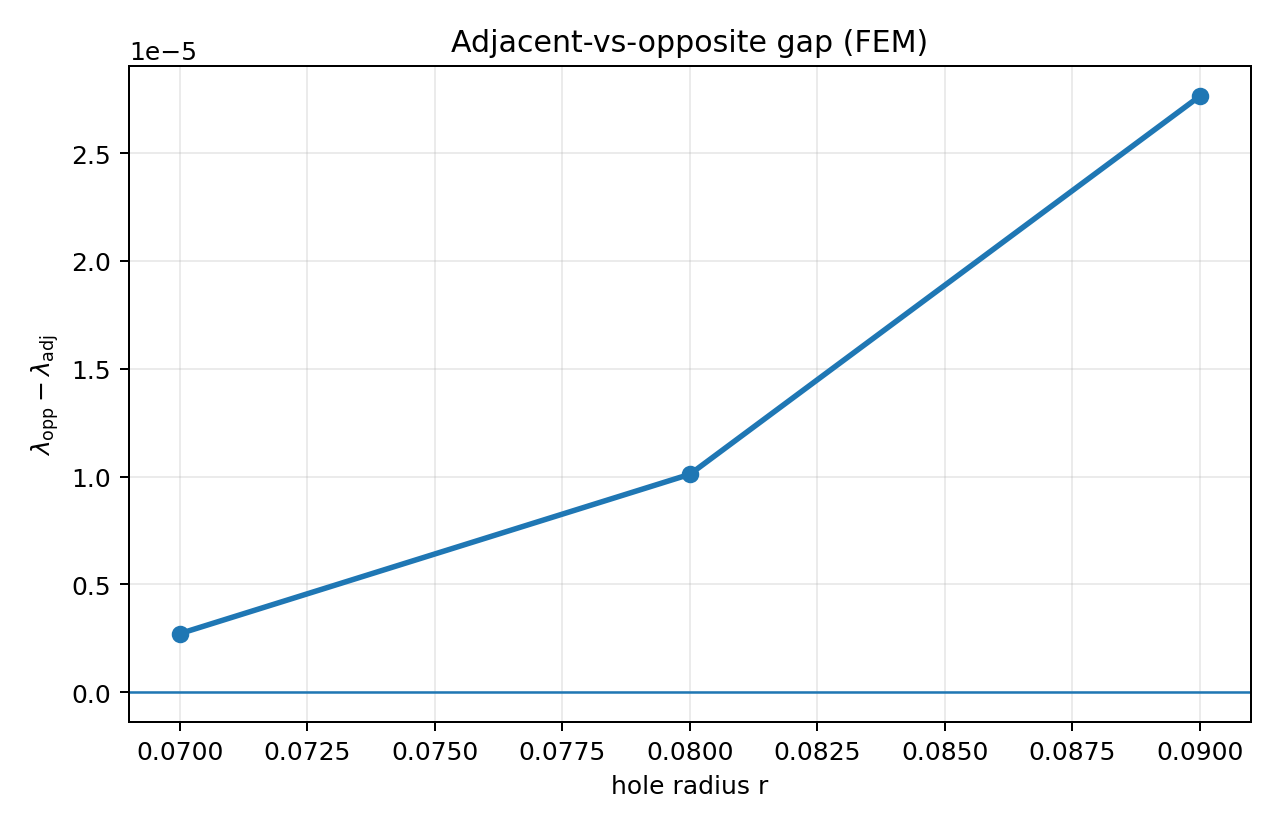}
\end{minipage}
\caption{Left: FEM comparison of the four representative two-hole branches. Right: the isolated positive gap $\lambda_{\mathrm{opp}}-\lambda_{\mathrm{adj}}$, showing that the adjacent-corner branch stays below the opposite-corner branch across the sampled radii.}
\label{fig:fem-branch-ordering-v13}
\end{figure}

The dataset also samples the same-corner cross-axis contact family
\[
(a_1,b_1)=(a,1),\qquad (a_2,b_2)=(1,b),\qquad (a-1)^2+(b-1)^2=4,
\]
at $r=0.08$. Table~\ref{tab:fem-contact-scan-v14} records a representative excerpt. The minimum sampled value occurs near $\theta=0.8942$ and remains well above both the adjacent-corner and opposite-corner values at the same radius.

\begin{table}[htbp]
\centering
\caption{Excerpt from the sampled same-corner contact-family scan at $r=0.08$.}
\label{tab:fem-contact-scan-v14}
\begin{tabular}{ccc}
\toprule
$\theta$ & $(a,b)$ & $\lambda_1$ \\
\midrule
$0.3500$ & $(2.8787,1.6858)$ & $5.032725$ \\
$\pi/4$ & $(2.4142,2.4142)$ & $5.032154$ \\
$0.8942$ & $(2.2522,2.5595)$ & $5.032120$ \\
$1.2208$ & $(1.6858,2.8787)$ & $5.032724$ \\
\bottomrule
\end{tabular}
\end{table}

At the sampled minimum one has $\lambda_1\approx 5.032120$, still well above the adjacent-corner value $5.017574$ and also above the opposite-corner value $5.017584$. In particular, the sampled contact family does not come close to challenging the corner branches. Overall, the numerical package supports the same qualitative conclusion as the analytic proof: among the representative two-hole branches considered here, the adjacent corner-tangent configuration is the asymptotically minimizing one.

\section{Distinct-corner asymptotics and completion of the proof}\label{sec:distinct-corner-proof}

Returning to the analytic proof, Section~\ref{sec:same-corner} leaves only the case when the two holes lie in two \emph{distinct} corner layers. This section proves the additive asymptotic expansion for that branch and then assembles the proof of Theorem~\ref{thm:main-small-hole-minimizer-v12}. Since the two corners are then separated by a fixed positive distance, the leading small-hole coefficient is additive.

\begin{proposition}[Separated-corner additivity]\label{prop:separated-corner-additivity-v12}
Fix two distinct corners $p\neq q$ of $Q$. For each of them choose local corner-scale parameters $(a_p,b_p)$ and $(a_q,b_q)$ in a fixed compact subset of $[1,\infty)^2$, and place one hole in the $O(r)$-layer of $p$ and the other in the $O(r)$-layer of $q$ accordingly. Then, uniformly on compact parameter sets,

\[
\lambda_1
=
\frac{\pi^2}{2}
+
\frac{\pi^4}{16}\bigl(\mathcal C(a_p,b_p)+\mathcal C(a_q,b_q)\bigr)r^4
+
o(r^4).
\]

In particular, within any fixed distinct-corner branch, the leading coefficient is minimized exactly when both holes are corner-tangent.
\end{proposition}

\begin{proof}

Let $K_{p,r}$ and $K_{q,r}$ be the two physical disks and set
\[
A_{p,r}:=\operatorname{Cap}_{Q}(K_{p,r},u_0),\qquad
A_{q,r}:=\operatorname{Cap}_{Q}(K_{q,r},u_0),\qquad
A_{pq,r}:=\operatorname{Cap}_{Q}(K_{p,r}\cup K_{q,r},u_0).
\]
Odd reflection across the two sides meeting at either corner and Lemma~\ref{lem:full-capacity-blowup} give, uniformly on compact parameter sets,
\[
A_{p,r}=\frac{\pi^4}{16}\mathcal C(a_p,b_p)r^4+o(r^4),\qquad
A_{q,r}=\frac{\pi^4}{16}\mathcal C(a_q,b_q)r^4+o(r^4).
\]
We claim
\begin{equation}\label{eq:separated-capacity-additivity}
A_{pq,r}=A_{p,r}+A_{q,r}+o(r^4).
\end{equation}
Choose disjoint relative neighborhoods of $p$ and $q$ in $\overline Q$ and smooth cutoffs $\eta_p,\eta_q$ supported there, equal to one near the corresponding corner, with $\eta_p^2+\eta_q^2\le1$. Let $V_{p,r}$ be the capacitary potential of $K_{p,r}$. Since $A_{p,r}=O(r^4)$, the family $r^{-2}V_{p,r}$ is bounded in $H_0^1(Q)$. Every compact subset of $Q$ is disjoint from $K_{p,r}$ for all sufficiently small $r$, so every weak limit is harmonic throughout $Q$. Since the limit also belongs to $H_0^1(Q)$, testing its weak harmonic equation with the limit itself shows that it is zero. Compactness gives
\[
r^{-2}V_{p,r}\to0\quad\text{strongly in }L^2(Q),
\]
and likewise for $V_{q,r}$.
For any fixed smooth cutoff $\eta$,
\[
\int_Q|\nabla(\eta V)|^2
=\int_Q\eta^2|\nabla V|^2+2\int_Q\eta V\nabla\eta\!\cdot\!\nabla V+\int_QV^2|\nabla\eta|^2.
\]
The last two terms are $o(r^4)$ for these potentials. Therefore the disjointly supported competitor
$\eta_pV_{p,r}+\eta_qV_{q,r}$ yields
\[
A_{pq,r}\le A_{p,r}+A_{q,r}+o(r^4).
\]
Conversely, let $V_{pq,r}$ be the union capacitary potential. The preceding upper bound gives $\|\nabla V_{pq,r}\|_2=O(r^2)$, and the same harmonic-limit argument yields $r^{-2}V_{pq,r}\to0$ in $L^2(Q)$. Since $\eta_pV_{pq,r}$ and $\eta_qV_{pq,r}$ are admissible for the two individual capacities,
\[
A_{p,r}+A_{q,r}
\le\int_Q|\nabla(\eta_pV_{pq,r})|^2+\int_Q|\nabla(\eta_qV_{pq,r})|^2
\le A_{pq,r}+o(r^4),
\]
where $\eta_p^2+\eta_q^2\le1$ is used in the last step. This proves \eqref{eq:separated-capacity-additivity}. The spaces $H_0^1(Q\setminus(K_{p,r}\cup K_{q,r}))$ converge to $H_0^1(Q)$ in the Mosco sense by the same recovery/weak-limit argument as in Lemma~\ref{lem:existence-continuity}, since the two disks concentrate to the finite boundary set $\{p,q\}$ of zero planar capacity.

Proposition~\ref{prop:boundary-simple-reduction} now gives the asserted eigenvalue expansion. The final statement follows from the coordinatewise strict monotonicity of $\mathcal C(a,b)$ in Corollary~\ref{cor:C-coordinatewise-monotone}.

\end{proof}

\begin{theorem}[Distinct-corner branch]\label{thm:distinct-corner-branch-v12}
Among all configurations with the two holes in two distinct corner layers, the unique asymptotic minimizer is the adjacent corner-tangent configuration, up to the dihedral symmetries of $Q$ and the interchange of the two holes.
\end{theorem}

\begin{proof}
By Proposition~\ref{prop:separated-corner-additivity-v12}, the leading coefficient in any distinct-corner branch is

\[
\mathcal C(a_p,b_p)+\mathcal C(a_q,b_q),
\]
so Corollary~\ref{cor:C-coordinatewise-monotone} forces both local parameter pairs to converge to $(1,1)$ along any minimizing sequence. Thus it remains only to decide whether the two limiting corners are adjacent or opposite. If they are opposite, Corollary~\ref{cor:opposite-layers-polarize} reflects one entire compact corner layer to an adjacent one and strictly lowers the first eigenvalue for all sufficiently small $r$. Hence an asymptotically minimizing distinct-corner sequence must use adjacent corners, and both disks are asymptotically corner-tangent.

\end{proof}

\begin{proof}[Proof of Theorem~\ref{thm:main-small-hole-minimizer-v12}]
Let $(x_{1,r},x_{2,r})\in\mathcal C_r$ be a minimizing configuration.

An adjacent corner-tangent competitor has eigenvalue excess $O(r^4)$ by Proposition~\ref{prop:separated-corner-additivity-v12}. Hence the minimizing configuration satisfies
\[
\Lambda_r(x_{1,r},x_{2,r})-\frac{\pi^2}{2}=O(r^4).
\]
Corollary~\ref{cor:two-hole-localization} therefore supplies a fixed $R>1$ such that, for all sufficiently small $r$, each center belongs to one of the four compact layers $\mathcal N_r^R(p)$. This is the global step that excludes the intermediate corner approaches not covered by the original branchwise argument. It also explicitly excludes the mixed ``side plus corner'' and ``two open-side'' configurations: an obstacle remaining in an open-side region has an individual cost of order at least $r^2$, incompatible with the $O(r^4)$ adjacent-corner competitor.

There are then only two possibilities.

\smallskip
\noindent\emph{Case 1: both holes lie in the same corner layer.}
This is excluded by Theorem~\ref{thm:same-corner-branch-complete-v11}, which shows that no same-corner configuration can be asymptotically minimizing.

\smallskip
\noindent\emph{Case 2: the two holes lie in two distinct corner layers.}
Then Theorem~\ref{thm:distinct-corner-branch-v12} applies and shows that the unique asymptotic minimizer in this branch is the adjacent corner-tangent pair.

Therefore, after relabeling and applying a symmetry of the square if needed,
\[
x_{1,r}=(-1+r,1-r)+o(r),
\qquad
x_{2,r}=(1-r,1-r)+o(r),
\]
as claimed.
\end{proof}

\section{Conclusion}\label{sec:conclusion}

We have shown that, in the small-hole regime, the first Dirichlet eigenvalue of the square with two equal hard circular obstacles is asymptotically minimized by the adjacent corner-tangent configuration, uniquely up to the dihedral symmetries of the square and interchange of the two holes. The proof combines five ingredients: a uniform global localization to $O(r)$ corner layers, the analysis of the side-tangent one-hole branch, an exact polarization comparison between adjacent and opposite corner pairs, the reduction and exclusion of same-corner clusters, and an additive asymptotic expansion for defects concentrating at two distinct corners.

The finite element validation in Section~\ref{sec:numerics} is consistent with this branch ordering and is fully reproducible through the archived benchmark dataset~\cite{Zhang2026DatasetFEM}. The methods developed here suggest several natural extensions, including more than two equal obstacles, other polygonal ambient domains, larger-radius regimes beyond the small-hole asymptotics, and Robin or Schr\"odinger variants in which the local corner model changes while the branch structure remains comparable.

A higher-dimensional analogue should preserve the basic localization principle---low eigenvalue excess is governed by the vanishing order of the unperturbed ground state and the capacity of the rescaled obstacle---but it is not a formal extension of the planar proof. In dimensions $d\ge3$ the relevant cell problem uses Newtonian rather than logarithmic/planar capacity, corner homogeneities and reflection geometries change, and the same-corner interaction and polarization steps require separate estimates. We therefore leave the higher-dimensional classification as a natural continuation rather than claim that the present two-dimensional argument carries over verbatim.

\section*{Statements and Declarations}
\noindent\textbf{Funding} This work was partially supported by the University of Ostrava under Grant No. SGS05/\allowbreak P\v{r}F/\allowbreak 2026. Yifan Zhang is also co-funded by the European Union under the REFRESH--Research Excellence For REgion Sustainability and High-tech Industries project number CZ.10.03.01/\allowbreak 00/\allowbreak 22\_003/\allowbreak 0000048 via the Operational Programme Just Transition.

\smallskip
\noindent\textbf{Competing interests} The authors have no competing interests to declare that are relevant to the content of this article.

\smallskip
\noindent\textbf{Author contributions} All authors contributed to the conception and development of the study. Yifan Zhang carried out the numerical computations and prepared the benchmark dataset. All authors contributed to the writing, revision, and approval of the final manuscript.

\smallskip
\noindent\textbf{Data availability} The finite element benchmark data supporting the numerical validation are publicly available in \href{https://doi.org/10.5281/zenodo.19338826}{https://doi.org/10.5281/zenodo.19338826}.

\smallskip
\noindent\textbf{Code availability} The reproducible code used for the finite element validation is publicly available in the same Zenodo repository cited in \href{https://doi.org/10.5281/zenodo.19338826}{https://doi.org/10.5281/zenodo.19338826}.

\smallskip
\noindent\textbf{Ethics approval} Not applicable.

\smallskip
\noindent\textbf{Consent to participate} Not applicable.

\smallskip
\noindent\textbf{Consent for publication} Not applicable.


\end{document}